\documentclass[11pt,twoside]{article}

\usepackage{amsfonts,epsfig,graphicx}
\usepackage{amsmath,amssymb,amsthm}
\usepackage{fullpage}

\usepackage{epsf}
\usepackage{fancyheadings}
\usepackage{graphics}
\usepackage{amsfonts}
\usepackage{amsmath}
\usepackage{psfrag}
\usepackage{mathrsfs}
\usepackage{accents}
\usepackage{authblk}
\usepackage{blindtext}

\usepackage{color}

\usepackage{multirow,tabularx}
\usepackage{hhline}

\newcolumntype{L}[1]{>{\raggedright\let\newline\\\arraybackslash\hspace{0pt}}m{#1}}
\newcolumntype{C}[1]{>{\centering\let\newline\\\arraybackslash\hspace{0pt}}m{#1}}
\newcolumntype{R}[1]{>{\raggedleft\let\newline\\\arraybackslash\hspace{0pt}}m{#1}}

\theoremstyle{plain}

\newtheorem{theo}{Theorem}[section]

\newtheorem{lem}{Lemma}[section]
\newtheorem{prop}{Proposition}[section]
\newtheorem{cor}{Corollary}[section]

\theoremstyle{definition} 

\newtheorem{nota}{Notation}[section]
\newtheorem{de}{Definition}[section]
\newtheorem{exa}{Example}[section]
\newtheorem{as}{Assumption}[section]
\newtheorem{alg}{Algorithm}[section]

\newcommand{\btheo}{\begin{theo}}
\newcommand{\bde}{\begin{de}}
\newcommand{\ble}{\begin{lem}}
\newcommand{\bpr}{\begin{prop}}
\newcommand{\bno}{\begin{nota}}
\newcommand{\bex}{\begin{exa}}
\newcommand{\bcor}{\begin{cor}}
\newcommand{\spro}{\begin{proof}}
\newcommand{\bas}{\begin{as}}
\newcommand{\balg}{\begin{alg}}

\newcommand{\etheo}{\end{theo}}
\newcommand{\ede}{\end{de}}
\newcommand{\ele}{\end{lem}}
\newcommand{\epr}{\end{prop}}
\newcommand{\eno}{\end{nota}}
\newcommand{\eex}{\end{exa}}
\newcommand{\ecor}{\end{cor}}
\newcommand{\fpro}{\end{proof}}
\newcommand{\eas}{\end{as}}
\newcommand{\ealg}{\end{alg}}

\theoremstyle{plain}

\newtheorem{theos}{Theorem}
\newtheorem{props}{Proposition}
\newtheorem{lems}{Lemma}
\newtheorem{cors}{Corollary}

\theoremstyle{definition}
\newtheorem{exas}{Example}
\newtheorem{algs}{Algorithm}
\newtheorem{asss}{Assumption}
\newtheorem{defns}{Definition}

\newcommand{\btheos}{\begin{theos}}
\newcommand{\etheos}{\end{theos}}
\newcommand{\bprops}{\begin{props}}
\newcommand{\eprops}{\end{props}}
\newcommand{\bdes}{\begin{defns}}
\newcommand{\edes}{\end{defns}}
\newcommand{\blems}{\begin{lems}}
\newcommand{\elems}{\end{lems}}
\newcommand{\bcors}{\begin{cors}}
\newcommand{\ecors}{\end{cors}}
\newcommand{\bexs}{\begin{exas}}
\newcommand{\eexs}{\end{exas}}
\newcommand{\balgs}{\begin{algs}}
\newcommand{\ealgs}{\end{algs}}
\newcommand{\bass}{\begin{asss}}
\newcommand{\eass}{\end{asss}}

\newlength{\widebarargwidth}
\newlength{\widebarargheight}
\newlength{\widebarargdepth}
\DeclareRobustCommand{\widebar}[1]{%
  \settowidth{\widebarargwidth}{\ensuremath{#1}}%
  \settoheight{\widebarargheight}{\ensuremath{#1}}%
  \settodepth{\widebarargdepth}{\ensuremath{#1}}%
  \addtolength{\widebarargwidth}{-0.3\widebarargheight}%
  \addtolength{\widebarargwidth}{-0.3\widebarargdepth}%
  \makebox[0pt][l]{\hspace{0.3\widebarargheight}%
    \hspace{0.3\widebarargdepth}%
    \addtolength{\widebarargheight}{0.3ex}%
    \rule[\widebarargheight]{0.95\widebarargwidth}{0.1ex}}%
  {#1}}

\newcommand{\Prob}{\ensuremath{\mathbb{P}}}
\newcommand{\Exs}{\ensuremath{\mathbb{E}}}

\newcommand{\MODEL}{\ensuremath{\mathbb{M}}}

\newcommand{\ubar}[1]{\underaccent{\bar}{#1}}
\newcommand{\uepsilon}{\ensuremath{\varepsilon}}
\newcommand{\lepsilon}{\ensuremath{\ubar{\varepsilon}}}

\newcommand{\thetastar}{\ensuremath{{\theta^\ast}}}
\newcommand{\lambdastar}{\ensuremath{{\lambda^\ast}}}

\makeatletter
\long\def\@makecaption#1#2{
        \vskip 0.8ex
        \setbox\@tempboxa\hbox{\small {\bf #1:} #2}
        \parindent 1.5em  
        \dimen0=\hsize
        \advance\dimen0 by -3em
        \ifdim \wd\@tempboxa >\dimen0
                \hbox to \hsize{
                        \parindent 0em
                        \hfil 
                        \parbox{\dimen0}{\def\baselinestretch{0.96}\small
                                {\bf #1.} #2
                                } 
                        \hfil}
        \else \hbox to \hsize{\hfil \box\@tempboxa \hfil}
        \fi
        }
\makeatother

\newcommand{\abs}[1]{\left\vert#1\right\vert}

\newcommand{\norm}[1]{\left\Vert#1\right\Vert}

\newcommand \bbP{\mathbb{P}}
\newcommand \bbE{\mathbb{E}}
\newcommand \bbR{\mathbb{R}}
\newcommand \bbL{\mathbb{L}}
\newcommand \bbH{\mathbb{H}}

\newcommand{\m}[1]{\mathcal{#1}}
\newcommand{\mb}[1]{\mathbb{#1}}

\begin{document}

\title{\bf{Bayesian model selection consistency and oracle inequality with intractable marginal likelihood}}

\author[1]{Yun Yang\thanks{Corresponding Author: yyang@stat.fsu.edu}}
\author[1]{Debdeep Pati\thanks{debdeep@stat.fsu.edu}}
\affil[1]{Department of Statistics, Florida State University}

\maketitle

\begin{abstract}
In this article, we investigate large sample properties of model selection procedures in a general Bayesian framework when a closed form expression of the marginal likelihood function is not available or a local asymptotic quadratic approximation of the log-likelihood function does not exist. Under appropriate identifiability assumptions on the true model, we provide sufficient conditions for a Bayesian model selection procedure to be consistent and obey the Occam's razor phenomenon, i.e., the probability of selecting the ``smallest" model that contains the truth tends to one as the sample size goes to infinity.  In order to show that a Bayesian model selection procedure selects the smallest model containing the truth, we impose a prior anti-concentration condition, requiring the prior mass assigned by large models to a neighborhood of the truth to be sufficiently small. In a more general setting where the strong model identifiability assumption may not hold, we introduce the notion of local Bayesian complexity and develop oracle inequalities for Bayesian model selection procedures. Our Bayesian oracle inequality characterizes a trade-off between the approximation error and a Bayesian characterization of the local complexity of the model, illustrating the adaptive nature of averaging-based Bayesian procedures towards achieving an optimal rate of posterior convergence. Specific applications of the model selection theory are discussed in the context of  high-dimensional nonparametric regression and density regression where the regression function or the conditional density is assumed to depend on a fixed subset of predictors.  As a result of independent interest,  we propose a general technique for obtaining upper bounds of certain small ball probability of stationary Gaussian processes. 
\end{abstract}


\section{Introduction}
A Bayesian framework offers a flexible and natural way to conduct model selection by placing prior weights over different models and using the posterior distribution to select a best one. However, unlike penalization-based model selection methods, there is a lack of general theory understanding large sample properties of Bayesian model selection procedures from a frequentist perspective. As a motivating example, we consider the problem of selecting a model from a  sequence of nested models. Under this special example, the Occam's Razor \cite{berger1996intrinsic} phenomenon suggests that a good model selection procedure is expected to select the smallest model that contains the truth. This example motivates us to investigate the consistency of a Bayesian model selection procedure, that is, whether the posterior tends to concentrate all its mass on the smallest model space that contains the true data generating model. 

In the frequentist literature, most model selection methods are based on optimization, where penalty terms are incorporated to penalize models with higher complexity. A large volume of the literature focuses on excess risk bounds and oracle inequalities, which are characterized via either some global measures of model complexity \cite{Vapnik1998,wegkamp2003,Bartlett:2003,Bousquet:2002,koltchinskii2006} that typically yield a suboptimal ``slow rate", or some improved local measures of the complexity \cite{koltchinskii2006,bartlett2005,Lever:2013} that yield an optimal ``fast rate" \cite{Bar:2007,Oneto2016}. An overwhelming amount of recent literature on penalization methods for high-dimensional statistical problems can also be analyzed under the model selection perspective. For example, in high dimensional linear regression, the famous Lasso~\cite{Tibshirani1996} places an $\ell_1$ penalty to induce sparsity, which can be considered as selecting a model from the model space consisting a sequence of $\ell_1$ balls with increasing radius; in sparse additive regression, by viewing the model space as all additive function spaces involving different subset of covariates with each univariate function lying in a univariate Reproducing kernel Hilbert space (RKHS) with increasing radius. \cite{Raskutti2012} proposed a minimax-optimal penalized method with a penalty term proportional to the sum of empirical norms and RKHS norms of univariate functions.

In the classical literature of Bayesian model selection in low dimensional parametric models, most results on model selection consistency relies on the critical property that the log-likelihood function can be locally approximated by a quadratic form of the parameter (such as the local asymptotic normality property) under a set of regularity assumption in the asymptotic regime when the sample size $n$ tends to infinity.
There is a growing body of literature which has provided some theoretical understanding of Bayesian variable selection for linear regression with a growing number of covariates, which is a special case of the model selection. In the moderate-dimension scenario (the number $p$ of covariates is allowed to grow with the sample size, but $p\leq n$),
\cite{Shang2011} established variable selection consistency in a Bayesian linear model, meaning that the posterior probability of the true model that contains all influential covariates tends to one as $n$ grows to infinity.   
\cite{johnson2012nonlocal} showed a selection inconsistency phenomenon for using several commonly used mixture priors, including local mixture (point mass at zero and slab prior
with non-zero value at null-value 0 of the slab density) priors, when $p$ is larger than the order of $\sqrt{n}$. To address this, they advocated the use of a non-local mixture priors (slab density has value $0$ at null value $0$) and obtained selection consistency when the dimension $p$ is $O(n)$. \cite{castillo2015bayesian} provided several conditions on the design matrix and the minimum signal strength  to ensure selection consistency with local priors  when $p \gg n$.  \cite{Narisetty2014} considered selection consistency using  a spike and slab local prior in  a high-dimensional scenario where $p$ can grow nearly exponentially with $n$. \cite{Yang2015} showed  variable selection consistency of high-dimensional Bayesian linear regression, where a prior is directly placed over the model space that penalizes each covariate in the model by a factor of $p^{-O(1)}$; in this setting, they showed a particular Markov chain Monte Carlo algorithm for sampling from the model space is rapidly mixing, meaning that the number of iterations required for the chain to converge to an $\varepsilon$-distance of stationary distribution from any initial configuration is at most polynomial in $(n, p)$. Although the aforementioned results on variable selection consistency in Bayesian linear models are promising, their proofs are all based on analyzing  closed form expressions of the marginal likelihood function, that is, the  likelihood function integrated with respect to the conditional prior distribution of the parameters given the model. 

The assumption of either an existence of a closed form expression of the marginal likelihood function or the existence of a local asymptotic quadratic approximation of the log-likelihood function significantly impedes the applicability of the current proof techniques to general model selection problems. For example, this assumption precludes the case when the parameter space is infinite dimensional, such as a space of functions or conditional densities indexed by predictors. To the best of our knowledge, little is known about the model selection consistency of an infinite dimensional Bayesian model or  in general when the marginal likelihood is intractable. The only relevant work in this direction is \cite{ghosal2008}, where they considered Bayesian nonparametric density estimation under an unknown regularity parameter, such as the smoothness level, which serves as the model index. In this setting, they showed that the posterior distribution tends to give negligible weight to models that are bigger than the optimal one, and thus selects the optimal model or smaller models that also approximate the true density well.

The goal of the current paper is to build a general theory for studying large sample properties of Bayesian model selection procedures, for example, model selection consistency and oracle inequalities. We show that in the Bayesian paradigm, the local Bayesian complexity, defined as $n^{-1}$ times the negative logarithm of the prior probability mass assigned to certain Kullback-Leibler divergence ball around the true model, plays the same role as the local complexity measures in oracle inequalities of penalized model selection methods. For example, when the conditional prior within each model is close to a ``uniform" distribution over the parameter space, the local Bayesian complexity becomes similar to a local covering entropy, recovering classical results \cite{lecam1973} by Le Cam for charactering rate of convergence using local entropy conditions.
In the special case of parametric models, when the prior is thick at the truth (which is a common assumption made in Bernstein-von Mises type results, refer to \cite{vdV00}), the local Bayesian complexity is $O(p\,\log n)$, which scales linearly in the dimension $p$ of the parameter space, recovering classical asymptotic theory on Bayesian model selection using the Bayesian information criterion (BIC).
In this article, we build oracle inequalities for Bayesian model selection procedures using the notion of local Bayesian complexity. Our oracle inequality implies that by properly distributing prior mass over different models, the resulting posterior distribution adaptively allocates its mass to models with the optimal rate of posterior convergence, revealing the adaptive nature of averaging-based Bayesian procedures.

Under an appropriate identifiability assumption on the true model, we show that a Bayesian model selection procedure is consistent, that is, the probability of selecting the ``smallest" model that contains the truth tends to one as $n$ goes to infinity. Here, the size of a model is determined by its local Bayesian complexity. In concrete examples, in order to show that a Bayesian model selection procedure tends to select the model that contains the truth and is smallest in the physical sense (for example, in variable selection case, the smallest model is the one exactly contains all influential covariates), we impose a prior anti-concentration condition, requiring the prior mass assigned by large models to a neighborhood of the truth to be sufficiently small. As a result of independent interest, in our proof for variable selection consistency of Bayesian high dimensional nonparametric regression using Gaussian process (GP) priors, we propose a general technique for obtaining upper bounds of certain small ball probability of stationary GP. The results complement the lower bound results obtained in \cite{kuelbs1993metric,li2001gaussian,vandervaart2009}.

Our results reveal that in the framework of Bayesian model selection, averaging based estimation procedure can gain advantage over optimization based procedures \cite{bhattacharya2016bayesian} in that i) the derivation on the convergence rate is simpler as the expectation exchanges with integration;  elementary probability inequalities such as Chebyshev's inequality and Markov's inequality can be used as opposed to more sophisticated empirical process tools for analysing optimization-based estimators; ii) the averaging based approach offers a more flexible framework to incorporate additional information and achieve adaptation to unknown hyper or tuning parameters, due to its average case analysis nature, which is different from the worse case analysis  of the optimization based approach. Overall, our results indicate that a Bayesian model selection approach naturally penalizes larger models since the prior distribution becomes more dispersive and the prior mass concentrating around the true model diminishes, manifesting Occam's Razor phenomenon. This renders a Bayesian approach to be naturally rate-adaptive to the best model by optimally trading off between goodness-of-fit and model complexity. 

The remainder of this paper is organized as follows.  In \S  \ref{sec:notn}, we introduce notations to be used in the subsequent sections.  In \S \ref{sec:prob},  we introduce the background and formulate the model selection problem.   The assumptions required for optimum posterior contraction rate are discussed in \S \ref{Sec:ContractionRate} with corresponding PAC-Bayes bounds in \S \ref{Sec:PACBayes}.  The main results are stated in \S \ref{sec:main} with the Bayesian model selection consistency theorems in \S \ref{Sec:BMSC} and Bayesian oracle inequalities in \S \ref{Sec:BOI}.  In \S \ref{Sec:NonReg} and \S \ref{Sec:DenReg}, we discuss applications of the model selection theory in the context of high-dimensional nonparametric regression and density regression where the regression function or the conditional density is assumed to depend on a fixed subset of predictors.

\subsection{Notations}\label{sec:notn}
Let $h(p,q) =(\int(p^{1/2}- q^{1/2})^2d\mu)^{1/2}$ and $D(p,q) = \int p\log$ $(p/q)d\mu$ stand for the Hellinger distance and Kullback-Leibler divergence, respectively, between two probability density functions $p$ and $q$ relative to a 
a common dominating measure $\mu$. We define an additional discrepancy measure $V(f,g) = \int f| \log(f/g) - D(f,g)|^2d\mu$. 
 For any $\alpha \in (0, 1)$, let
\begin{align}\label{eq:renyi_def}
D^{(n)}_\alpha(p, q) = \frac{1}{\alpha-1} \log \int p^{\alpha} q^{1 - \alpha} d\mu 
\end{align}
denote the R{\'e}nyi divergence of order $\alpha$. Let us also denote by $A_{\alpha}(p, q)$  the quantity $\int p^{\alpha} q^{1 - \alpha} d\mu = e^{(\alpha-1) D^{(n)}_\alpha(p, q)}$, which we shall refer to as the $\alpha$-affinity. When $\alpha = 1/2$, the $\alpha$-affinity equals the Hellinger affinity. Moreover, $0 \le A_{\alpha}(p, q) \le 1$ for any $\alpha \in (0, 1)$, implying that $D^{(n)}_\alpha(p, q) \ge 0$ for any $\alpha \in (0, 1)$ and equality holds if and only if $p\equiv q$.  Relevant  inequalities  and properties related to 
R{\'e}nyi divergence can be found in \cite{van2014renyi}. Let $N(\varepsilon,\, \mathcal{F},\,  d)$ denote the $\varepsilon$-covering number of the space $\mathcal{F}$ with respect to a semimetric $d$.  Operator ``$\lesssim$'' denotes less or equal up to a multiplicative positive constant relation.   For a finite set $A$, let $|A|$ denote the cardinality of $A$.
The set of natural numbers is denoted by  
$\mathbb{N}$. The $m$-dimensional simplex is denoted by $\Delta^{m-1}$.  $I_k$ stands for the $k \times k$ identity matrix.
Let $\phi_{\mu,\sigma}$ denote a multivariate normal density with mean $\mu \in \mathbb{R}^k$ and covariance 
matrix $\sigma^2 I_k$.

\section{Background and problem formulation}\label{sec:prob}
Let $(X^{(n)},\m A^{(n)}, \mb P^{(n)}_\theta:\,  \theta \in \Theta)$ be a sequence of statistical experiments with observations $X^{(n)}=(X_1,X_2,\ldots,X_n)$, where $\theta$ is the parameter of interest living in an arbitrary parameter space $\Theta$, and $n$ is the sample size. Our framework allows the observations to deviate from identically or independently distributed setting (abbreviated as non-i.i.d.) \cite{ghosal2007}. For example, this framework covers Gaussian regression with fixed design, where observations are independent, but nonidentically distributed (i.n.i.d.). For each $\theta$, let $\bbP^{(n)}_\theta$ admit a density $p_{\theta}^{(n)}$ relative to a $\sigma$-finite measure $\mu^{(n)}$. Assume that $(x,\theta)\to p_{\theta}^{(n)}(x)$ is jointly measurable relative to $\mathcal{A}^{(n)} \otimes\mathcal{B}$, where $\mathcal{B}$ is a $\sigma$-field on $\Theta$. 

For a model selection problem, let $\MODEL=\{\MODEL_\lambda, \, \lambda\in \Lambda\}$ be the model space consisting of a countable number of models of interest. Here, $\Lambda$ is a countable index set, $\MODEL_\lambda=\{\Prob_{\theta}: \, \theta\in\Theta_\lambda\}$ is the model indexed by $\lambda$, and $\Theta_\lambda$ is the associated parameter space. Assume that the union of all $\Theta_\lambda$'s constitutes the entire parameter space $\Theta$, that is $\Theta = \bigcup_{\lambda\in \Lambda} \Theta_\lambda$. 
We consider the model selection problem in its full generality by allowing $\{\Theta_\lambda, \,\lambda\in \Lambda\}$ to be arbitrary. For example, they can be Euclidean spaces with different dimensions or spaces of functions depending on different subsets of covariates. Moreover, $\Theta_\lambda$'s may overlap or have inclusion relationships. 

We use the notation $\thetastar$ to denote the true parameter, also referred to as the truth, corresponding to the data generating model $\mb P^{(n)}_\thetastar$. Let $\lambdastar$ denote the index corresponding to the smallest model $\MODEL_\lambda$ that contains $\Prob^{(n)}_{\thetastar}$. 
More formally, a smallest model is defined as 
\begin{align}\label{Eqn:Model_assump}
\Theta_\lambdastar = \bigcap_{\lambda:\, \thetastar\in \Theta_\lambda} \Theta_\lambda.
\end{align}
Here, we have made an implicit assumption that for any $\theta^\ast \in\Theta$, there always exists a unique model $\MODEL_{\lambda^\ast}$ such that the preceding display is true.
This assumption rules out pathological cases where $\m P^{(n)}_\thetastar$ may belong to several incomparable models and a smallest model cannot be defined.

Let $\pi_\lambda$ be the prior weight assigned to model $\MODEL_\lambda$ and $\Pi_\lambda(\cdot)$ be the prior distribution over $\Theta_\lambda$ in $\MODEL_\lambda$, that is, $\Pi_\lambda(\cdot)$ is the conditional prior distribution of $\theta$ given model $\MODEL_\lambda$ being selected. 
Under this joint prior distribution $\Pi = \{(\pi_\lambda,\Pi_\lambda):\,\lambda\in\Lambda\}$ on $(\lambda, \theta)$, we obtain a joint posterior distribution using Bayes Theorem, 
\begin{align}\label{Eqn:normalpost}
\Pi( \theta \in B,\, \lambda \, |\, X^{(n)}) = \frac{ \pi_\lambda\, \int_{B} p^{(n)}_{\theta} (X^{(n)}) \, \Pi_\lambda(d\theta) }{\sum_{\lambda\in\Lambda} \pi_\lambda \,  \int_{\Theta_\lambda} p^{(n)}_{\theta} (X^{(n)})  \, \Pi_\lambda(d\theta)},\quad \forall B \in \m B.
\end{align}
By integrating this posterior distribution over $\lambda$ or $\theta$, we obtain respectively the marginal posterior distribution of the model index $\lambda$ or the parameter $\theta$.
In this paper, we also consider a class of quasi-posterior distributions obtained by using the $\alpha$-fractional likelihood 
\cite{walker2001bayesian,martin2016optimal,bhattacharya2016bayesian}, which is the usual likelihood raised to power $\alpha\in(0,1)$,
\begin{align*}
L_{n, \alpha}(\theta) = \Big[ p^{(n)}_\theta (X^{(n)}) \Big]^{\alpha}.
\end{align*}
Let $\Pi_{n, \alpha}(\cdot)$ denote the qausi-posterior distribution, also referred to as the $\alpha$-fractional posterior distribution, obtained by combining the fractional likelihood $L_{n, \alpha}$ with the prior $\Pi$,
\begin{align}\label{Eqn:quasilpost}
\Pi_{\alpha}( \theta \in B,\, \lambda\, |\, X^{(n)}) = \frac{ \pi_\lambda\, \int_{B}\big(p^{(n)}_{\theta} (X^{(n)})\big)^\alpha \, \Pi_\lambda(d\theta) }{\sum_{\lambda\in\Lambda} \pi_\lambda \,  \int_{\Theta_\lambda} \big(p^{(n)}_{\theta} (X^{(n)}) \big)^\alpha \, \Pi_\lambda(d\theta)},\quad \forall B \in \m B.
\end{align}
The posterior distribution in~\eqref{Eqn:normalpost} is a special case of the fractional posterior with $\alpha=1$. For this reason, we also refer the posterior  in~\eqref{Eqn:normalpost}  as the regular posterior distribution.
As described in \cite{bhattacharya2016bayesian} (refer to Section~\ref{Sec:ContractionRate} of the current article for a brief review), the development of  asymptotic theory for fractional posterior distributions demands much simpler conditions than the regular posterior while maintaining the same rate of convergence. However, the  downside is two-fold  i) the credible intervals from the fractional posterior distribution maybe $\alpha^{-1}$ times wider than those from the normal posterior, at least for the regular parametric models where the Bernstein-von Mises theorem holds; ii) the simplified asymptotic results only apply for a certain class of distance measures $d_n$. The first downside can be remedied by post-processing the credible intervals, for example, reduce its width from the center by a factor of $\alpha$; and the second by considering a general class of risk function-induced fractional quasi-posteriors. We leave the latter as a topic of future research.  

\paragraph{Definition:} We say that a Bayesian procedure has  model selection consistency if
\begin{align*}
 \Exs_{\thetastar}^{(n)}[ \Pi_{\alpha}(\lambda = \lambda^\ast \, |\, X_1,\ldots,X_n) ] \rightarrow 1,\quad \mbox{as $n\to\infty$,}
\end{align*}
where either $\alpha=1$ or $\alpha\in(0,1)$, depending on whether regular or fractional posterior distribution is used.

Our general framework allows the truth $\thetastar$ to belong to multiple, even infinitely many $\Theta_\lambda$'s, for example, when $\MODEL$ is a sequence of nested models $\MODEL_1\subset\MODEL_2\subset\cdots$. In such a situation, the Occam's Razor principle suggests that a good statistical model selection procedure should be able to select the most parsimonious model $\MODEL_\lambda$ that fits the data well. This criterion is consistent with our definition of Bayesian model selection consistency, which requires the marginal posterior distribution over the model space to concentrate on the smallest model $\MODEL_\lambdastar$ that contains $\mb P^{(n)}_\thetastar$.
If a Bayesian procedure results in model selection consistency, then we can define a single selected model $\MODEL_{\widehat{\lambda}_{\alpha}}$, with its model index being selected as
\begin{align*}
\widehat{\lambda}_{\alpha}:\, = \arg\max_\lambda \Pi_\alpha(\lambda = \lambda^\ast \, |\, X_1,\ldots,X_n) ],
\end{align*}
which is the posterior mode over the index set. This model selection procedure satisfies the model selection consistency criterion under the frequentist perspective, i.e., 
\begin{align*}
\Prob_{\theta^\ast}^{(n)}\big[\widehat{\lambda}_{\alpha} = \lambda^\ast\big] \rightarrow 1,\quad \mbox{as $n\to\infty$.}
\end{align*}

\subsection{Contraction of posterior distributions}\label{Sec:ContractionRate}
In this subsection, we review the theory \cite{ghosal2000,ghosal2007,bhattacharya2016bayesian}  on the contraction rate of regular and fractional posterior distributions. 
For notational simplicity, we drop the dependence on the model index $\lambda$ in the current (\S \ref{Sec:ContractionRate}) and the next (\S \ref{Sec:PACBayes}) subsections, and the results can be applied to any $\MODEL_\lambda$ with $\lambda\in\Lambda$. 
Recall that the observations $X^{(n)}=(X_1,\ldots,X_n)$ are realizations from the data generating model $\Prob_{\theta^\ast}$, where the true parameter $\theta^\ast$ may or may not belong to the parameter space $\Theta$. When $\theta^\ast$ does not belong to $\Theta$, the model is considered to be misspecified. 


\paragraph{Regular posterior distribution:} We consider the regular posterior distribution 
\begin{align}\label{Eqn:normalpost_review}
\Pi( \theta \in B\, |\, X^{(n)}) = \frac{ \int_{B}  p^{(n)}_{\theta} (X^{(n)}) \, \Pi(d\theta) }{ \int_{\Theta} p^{(n)}_{\theta} (X^{(n)})  \, \Pi(d\theta)},\quad \forall B \in \m B.
\end{align}
We introduce three common assumptions that are adopted in the literature \cite{ghosal2000,ghosal2007}. Let $d_n$ be a semimetric on $\Theta$ to quantify the distance between $\theta$ and $\theta^\ast$.

\paragraph{Assumption A1 (Test condition):} There exist constants $a>0$ and $b>0$ such that for every $\varepsilon>0$ and $\theta_1\in \Theta$ with $d_n(\theta_1, \theta^\ast) > \varepsilon$, there exists a test $\phi_{n,\theta_1}$ such that
\begin{align}\label{Eqn:Test}
\Prob_{\theta^\ast}^{(n)} \phi_{n,\theta_1} \leq e^{-b \,n\,\varepsilon^2}\qquad \mbox{and}\qquad
\sup_{\theta\in\Theta:\, d_n(\theta,\, \theta_1)\leq \,a\, \varepsilon}\Prob_{\theta}^{(n)} (1-\phi_{n,\theta_1}) \leq e^{-b\,n\,\varepsilon^2}.
\end{align}

\noindent Assumption A1 ensures that $\theta^\ast$ is statistically identifiable, which guarantees the existence of a test function $\phi_{n, \theta_1,\lambda}$ for testing against parameters  close to $\theta_1$ in $\Theta$, and provides upper bounds for its Type I and II errors. In the special case when $X_i$'s are i.i.d.~observations and $d_n(\theta,\theta')$ is the Hellinger distance between $\Prob_{\theta}$ and $\Prob_{\theta'}$, such a test $\phi_{n,\theta_1}$ always exists \cite{ghosal2007}.

\paragraph{Assumption A2 (Prior concentration):}
There exist a constant $c > 0$ and a sequence $\{\uepsilon_{n}\}_{n=1}^\infty$ with  $n\, \uepsilon_{n}^2 \to\infty$ such that
\begin{align*}
\Pi (B_{n}(\theta^\ast,\uepsilon_{n})) \geq e^{-c \, n \, \uepsilon_{n}^2}.
\end{align*}
Here for any $\theta_0$, $B_{n}(\theta_0,\varepsilon)$ is defined as the following $\varepsilon$-KL neighbourhood in $\Theta$ around $\theta_0$,
\begin{align*}
B_{n}(\theta_0,\varepsilon)=\{\theta\in\Theta:\, D(p_{\theta_0}^{(n)}, p_{\theta}^{(n)})\leq n\,\varepsilon^2,\, V(p_{\theta_0}^{(n)}, p_{\theta}^{(n)})\leq n\,\varepsilon^2\}.
\end{align*}

When the model is well-specified, a common procedure to verify Assumption A2 is to show that it holds for all $\theta^\ast$ in $\Theta$.  This stronger version of the parameter prior concentration condition characterizes the compatibility of the prior distribution with the parameter space. In fact, this assumption together with Assumption A3 below implies that the prior distribution is almost  ``uniformly distributed'' over $\Theta$. In fact, if the covering entropy $\log N(\uepsilon_{n,\lambda}, \, \mathcal{F}_{n,\lambda}, \, d_n)$  is of the order $n\,\uepsilon_{n,\lambda}^2$, then there are roughly $e^{n\,\uepsilon_{n,\lambda}^2}$ disjoint $\uepsilon_{n,\lambda}$-balls in the parameter space. Assumption A2 requires each ball to receive mass $e^{-c\, n\,\uepsilon_{n,\lambda}^2}$, which matches up to a constant in the exponent of the average prior probability mass $e^{-n\,\uepsilon_{n,\lambda}^2}$ received by those disjoint $\uepsilon_{n,\lambda}$-balls.

\paragraph{Assumption A3 (Sieve sequence condition):} 
For some constant $D>1$, there exists a sequence of sieves $\mathcal{F}_{n} \subset \Theta$, $n=1,2, \ldots$, such that 
\begin{align*}
\log N(\varepsilon, \, \mathcal{F}_{n}, \, d_n) \leq n\,\varepsilon_n^2
\qquad\mbox{and}\qquad \Pi_\lambda(\mathcal{F}_{n}^c) \leq e^{-D\,n\,\varepsilon_n^2}.
\end{align*}

Assumption A3 allows to focus our attention to the most important region in the parameter space that is not too large, but still possesses most of the prior mass. Roughly speaking, the sieve $\mathcal{F}_{n}$ can be viewed as the effective support of the prior distribution at sample size $n$. The construction of the sieve sequence is required merely for the purpose of a proof and is not need for implementing the methods, which is different from sieve estimators.

\btheos[Contraction of the regular posterior~\cite{ghosal2007}]
\label{Thm:rate}
If Assumptions A1-A3 hold, then there exists a sufficiently large constant $M$ such that the posterior distribution~\eqref{Eqn:normalpost} satisfies
\begin{align*}
&\Exs_{\thetastar}^{(n)}[ \Pi( d_n(\theta, \,\thetastar) > M \,\uepsilon_{n}\, |\, X^{(n)}) ] \rightarrow 0,\quad\mbox{and}\\
& \Exs_{\thetastar}^{(n)}[ \Pi( \theta\in \mathcal{F}_{n} \, |\, X^{(n)}) ] \rightarrow 1,\quad \mbox{as $n\to\infty$.}
\end{align*}
\etheos

We say that Bayesian procedure for model $\MODEL$ has a posterior contraction rate at least $\uepsilon_{n}$ relative to $d_n$ if the first display in Theorem~\ref{Thm:rate} holds, or simply call $\uepsilon_{n,\lambda}$ as the posterior contraction rate relative to $d_n$ when no ambiguity occurs.
If we have an additional prior anti-concentration condition
\begin{align}\label{Eqn:Anti}
\Pi_\lambda (d_n(\theta,\, \theta^\ast) \leq \lepsilon_{n,\lambda}) \leq e^{- H \, n \, \varepsilon_{n,\lambda}^2},
\end{align}
where $\lepsilon_{n,\lambda} < \varepsilon_{n,\lambda}$ and $H$ is a sufficiently large constant, then by defining a new sieve sequence in Assumption A3 as
\begin{align*}
\widebar{\mathcal{F}}_{n}:\, =\mathcal{F}_{n} \cap \big\{\theta\in\Theta:\, d_n(\theta,\, \theta^\ast) \leq \lepsilon_{n}\big\},
\end{align*}
we have the following two-sided posterior contraction result, as a direct consequence of Theorem~\ref{Thm:rate}.

\bcors[Two sided contraction rate]\label{Corr:normpost}
Under the assumptions of Theorem~\ref{Thm:rate}, if \eqref{Eqn:Anti} is true for some sufficiently large $H$, then
\begin{align*}
\Exs_{\thetastar}^{(n)}[ \Pi( \theta:\, \lepsilon_{n} \leq d_n(\theta, \,\theta^\ast) \leq M \,\varepsilon_{n}\, |\, X^{(n)}) ] \rightarrow 1,\quad \mbox{as $n\to\infty$.}
\end{align*}
\ecors
As we will show in the next section, a prior anti-concentration condition similar to \eqref{Eqn:Anti} plays an important role for Bayesian model selection consistency in ensuring overly large models that contain $\mb P^{(n)}_\thetastar$ to receive negligible posterior mass as $n\to\infty$.

\

\paragraph{Fractional posterior distributions:} Now let us turn to the fractional posterior distribution below that is based on fractional likelihood \cite{walker2001bayesian,martin2016optimal,bhattacharya2016bayesian},
\begin{align}\label{Eqn:quasipost_review}
\Pi_\alpha( \theta \in B\, |\, X^{(n)}) = \frac{ \int_{B}  \big(p^{(n)}_{\theta} (X^{(n)})\big)^\alpha \, \Pi(d\theta) }{ \int_{\Theta} \big(p^{(n)}_{\theta} (X^{(n)})\big)^\alpha  \, \Pi(d\theta)},\quad \forall B \in \m B.
\end{align}
Observe that the fractional-likelihood implicitly penalizes all models that are far away from the true model $\Prob_{\theta^\ast}$ through the following identity,
\begin{align}\label{Eqn:Key_ind}
\mb E^{(n)}_\thetastar\bigg[\bigg(\frac{p^{(n)}_{\theta} (X^{(n)})}{p^{(n)}_{\thetastar} (X^{(n)})} \bigg)^\alpha \bigg] = \exp\big\{-D^{(n)}_\alpha \big(p^{(n)}_{\theta}, p^{(n)}_{\theta_0}\big) \big\}. 
\end{align}
Hence by dividing both the denominator and numerator in \eqref{Eqn:quasipost_review} with an $\theta$-independent quantity $\big(p^{(n)}_{\thetastar} (X^{(n)})\big)^\alpha$, we can get rid of the test condition  A1 and the sieve sequence condition A3 that are used to build a test procedure to discriminate  all far away models in the theory for the contraction of the regular posterior distribution. It is discussed in \cite{bhattacharya2016bayesian} that some control on the complexity of the parameter space is also necessary to ensure the consistency of the regular posterior distribution \cite{ghosal2000,ghosal2007}. Therefore, the fractional posterior distribution is, at least theoretically, more appealing than the regular posterior distribution, since the prior concentration condition A2 alone is sufficient to guarantee the contraction of the posterior (see Theorem \ref{Thm:quasirate} below). More discussions and comparisons can be found in \cite{bhattacharya2016bayesian}.
For notational convenience, we assume the shorthand $D^{(n)}_\alpha(\theta, \thetastar)$ for $D^{(n)}_\alpha(\Prob^{(n)}_\theta, \Prob^{(n)}_\thetastar)$. 

\btheos[Contraction of the fractional posterior \cite{bhattacharya2016bayesian}]
\label{Thm:quasirate}
If Assumption A2 holds, then there exists a sufficiently large constant $M$ independent of $\alpha$ such that the fractional posterior distribution~\eqref{Eqn:quasipost_review} with order $\alpha$ satisfies
\begin{align*}
&\Exs_{\thetastar}^{(n)}[ \Pi_{\alpha}( n^{-1}\, D^{(n)}_\alpha(\theta, \,\thetastar) > \frac{M}{1-\alpha} \, \uepsilon^2_{n}\, |\, X^{(n)}) ] \rightarrow 0,
\end{align*}
and for any subset $\mathcal{F}_{n}\subset \Theta$ satisfying $\Pi(\mathcal{F}_{n}^c) \leq e^{-D\,\alpha\,n\, \varepsilon_{n}^2}$, where $D$ is a sufficiently large constant independent of $\alpha$,
\begin{align*}
 \Exs_{\thetastar}^{(n)}[ \Pi^{(n)}_{\alpha}( \theta\in \mathcal{F}_{n} \, |\, X^{(n)}) ] \rightarrow 1,\quad \mbox{as $n\to\infty$.}
\end{align*}
\etheos

In the special case of i.i.d. observations, the $\alpha$-divergence becomes $D^{(n)}_\alpha(\theta,$ $\thetastar) = n\, D_\alpha(\theta, \,\thetastar)$, where $D_\alpha(\theta, \,\thetastar)$ is the $\alpha$-divergence for one observation. In this case, Theorem~\ref{Thm:quasirate} shows that a Bayesian procedure using the fractional posterior distribution has a posterior contraction rate of at least $\uepsilon_{n}$ relative to the average $\alpha$-R\'{e}nyi divergence $n^{-1}D^{(n)}_\alpha$. Similar to corollary~\eqref{Corr:normpost} for the regular posterior distribution, 
if we have an additional prior anti-concentration condition
\begin{align}\label{Eqn:quasiAnti}
\Pi (n^{-1}\,D^{(n)}_{\alpha}(\theta,\, \theta^\ast) \leq \lepsilon^2_{n}) \leq e^{- H\,\alpha \, n \, \varepsilon_{n}^2},
\end{align}
where $\lepsilon_{n} < \varepsilon_{n}$ and $H$ is a sufficiently large constant, then we have the following two-sided posterior contraction result as a corollary.

\bcors\label{Cor::quasirate}
Under the assumptions of Theorem~\ref{Thm:quasirate}, if \eqref{Eqn:quasiAnti} is also true for a sufficiently large constant $H > 0$, then for some sufficiently large constant $M$,
\begin{align*}
\Exs_{\thetastar}^{(n)}[ \Pi_{\alpha}( \theta:\, \lepsilon^2_{n} \leq n^{-1}\,D^{(n)}_\alpha(\theta, \,\theta^\ast) \leq \frac{M}{1-\alpha} \,\varepsilon^2_{n}\, |\, X^{(n)}) ] \rightarrow 1,\quad \mbox{as $n\to\infty$.}
\end{align*}
\ecors

\subsection{PAC-Bayes bound}\label{Sec:PACBayes}
Model selection is typically a more difficult task than parameter estimation. As we will see in the next section, Bayesian model selection consistency is stronger than the property that the (fractional) posterior achieves a certain rate of contraction. In fact,  Bayesian model selection consistency requires an additional identifiability condition (see Assumption B1 in the next section) on the true model $\MODEL_{\lambda^\ast}$ that assumes a proper gap between models that do not contain $\Prob_{\theta^\ast}$ and models  containing $\Prob_{\theta^\ast}$. However, a posterior contraction rate $\varepsilon_n$ is still attainable even when such misspecified models receive considerable posterior mass, as long as $\theta^\ast$ can be approximated by parameters in those models with $d_n$-error not exceeding $\varepsilon_n$. Therefore, a Bayesian procedure may still achieve the estimation optimality without model selection consistency. In our model selection framework, such a property can be best captured and characterized by a Bayesian version of the frequentist oracle inequality --- a PAC-Bayes type inequality \cite{McAllester1998,McAllester1999}. For convenience and simplicity, we only focus on the fractional posterior distribution~\eqref{Eqn:quasipost_review}, which only requires Assumption A2. Similar results also apply to the regular posterior distribution~\eqref{Corr:normpost} when more technical conditions such as Assumptions A1 and A3 are assumed.

Under the setup and notation in Section~\ref{Sec:ContractionRate}, a typical PAC-Bayes type inequality takes the form as
\begin{align*}
\int R(\theta,\thetastar) \Pi(d\theta\,|\,X^{(n)}) \leq \int S_n(\theta,\thetastar) \rho(d\theta) + \frac{1}{\kappa_n} D(\rho,\,\Pi)  + \mbox{Rem},
\end{align*}
for all probability measure $\rho$ that is absolutely continuous with respect to the prior $\Pi$. Here $R$ is the risk function, $S_n$ is certain function that measures the discrepancy between $\theta$ and $\thetastar$ on the support of the measure $\rho$, $\kappa_n$ is a tuning parameter and $ \mbox{Rem}$ is a remainder term. The PAC-Bayes inequality (\cite{bhattacharya2016bayesian}) we review is for the fractional posterior distribution~\eqref{Eqn:quasipost_review}, where the risk function $R$ is a multiple of the $\alpha$-Renyi divergence $D^{(n)}_\alpha$ in~\eqref{Eqn:Key_ind}, and $S_n(\theta,\theta_0)$ a multiple of the negative log-likelihood ratio between $\theta$ and $\theta^\ast$,
$$
r_n(\theta,\theta^\ast):\,=-\log\big(p^{(n)}_{\theta}(X^{(n)})/p^{(n)}_{\thetastar}(X^{(n)})\big).$$

\btheos[PAC Bayes inequality \cite{bhattacharya2016bayesian}]
\label{Thm:PAC-Bayes-one-model}
Fix $\alpha \in (0, 1)$. Then, 
\begin{equation}\label{eq:renyi_bd}
\begin{aligned}
 &\int \frac{1}{n}\, D^{(n)}_{\alpha}(\theta, \thetastar) \,\Pi_{\alpha}(d\theta\,|\,X^{(n)}) \le \frac{\alpha}{n (1 - \alpha)} \int r_n(\theta,\thetastar) \,\rho(d\theta) \\
 &\qquad\qquad\qquad\quad
   + \frac{1}{n(1-\alpha)}  \log(1/\varepsilon),\quad
  \mbox{$\forall$ probability measure $\rho\ll \Pi$},
\end{aligned}
\end{equation}
holds with $\Prob^{(n)}_{\theta^\ast}$ probability at least $(1 - \varepsilon)$.
\etheos

In particular, if we choose the measure $\rho$ to be the conditional prior $\rho_U(\cdot) = \Pi(\cdot \,|\, U)=\Pi(\cdot \cap U)/\Pi(U)$ with $U=B_{n}(\thetastar,\varepsilon_{n})$ being the KL neighborhood in Assumption A2, then we have the following corollary.

\bcors[Bayesian oracle inequality \cite{bhattacharya2016bayesian}]\label{cor:renyi}
Let $\varepsilon_n \in (0, 1)$ satisfy $n \varepsilon_n^2 > 2$ and fix $D > 1$. Then, with $\Prob^{(n)}_{\theta^\ast}$ probability at least $1 - 2/\{(D-1)^2 n \varepsilon_n^2\}$, 
\begin{align}\label{eq:cor_renyi}
 \int \frac{1}{n}\,D^{(n)}_{\alpha}(\theta, \thetastar) \Pi_{\alpha}(d\theta\,|\,X^{(n)}) \le \frac{(D + 1) \alpha}{1 - \alpha}  \varepsilon_n^2  + \Big\{ -\frac{1}{n (1 - \alpha)} \log \Pi(B_{n}(\thetastar,\varepsilon_{n})) \Big\}.
\end{align}
\ecors

\noindent Note the posterior expected risk always provides a upper bound on the estimation loss $n^{-1}\,D^{(n)}_\alpha(\widehat{\theta}_M,\thetastar)$ of the posterior mean $\widehat{\theta}_M:\,=\int \theta\,  \Pi_\alpha(d\theta\,|\,X^n)$, since the loss function $n^{-1}\,D^{(n)}_\alpha(\cdot,\thetastar)$ is convex in its first argument for any $\alpha\in(0,1)$.
Corollary~\ref{cor:renyi} shows that the posterior estimation risk is a trade-off between two terms: the term $\varepsilon_{n}^2$ due to the approximation error, and the term $-\frac{1}{n}\log \Pi(B_n(\thetastar,\varepsilon_n))$ characterizes the prior concentration at $\thetastar$. Similar to the discussion after Assumption A2, this second term can also be viewed as a measurement of the local covering entropy of the parameter space $\Theta$ around $\thetastar$ if the prior is  ``compatible'' to $\Theta$, and therefore reflects the local complexity of the parameter space. 

\paragraph{Definition:} For a model $\m M=\{\mb P_{\theta}^{(n)}:\,\theta\in\Theta\}$, we define its \emph{local Bayesian complexity} at parameter $\thetastar$ with radius $\varepsilon$ as
\begin{align*}
-\frac{1}{n}\log \Pi(B_{n}(\thetastar,\varepsilon)),
\end{align*}
where $B_{n}(\thetastar,\varepsilon)$ is given in Assumption A2. Here $\thetastar$ may or may not belong to $\Theta$.

\

According to Corollary~\ref{cor:renyi}, the Bayesian risk of an $\alpha$-fractional posterior is bounded up to a constant by the square critical radius $(\varepsilon^\ast_n)^2$, where $\varepsilon^\ast_n$ is the smallest solution of 
\begin{align*}
-\frac{1}{n}\log \Pi(B_{n}(\thetastar,\varepsilon))=\alpha\,\varepsilon^2
\end{align*}
(more discussions on the local Bayesian complexity can be found in \cite{bhattacharya2016bayesian}). Using this notation, Assumption A2 can be translated to the fact that the local Bayesian complexity $-\frac{1}{n}\log \Pi(B_{n}(\thetastar,\varepsilon_n))$ at truth $\thetastar$ with radius $\varepsilon_n$ is upper bounded by $c\,\varepsilon_n^2$. When we refer to a local Bayesian complexity without mentioning its radius, we implicitly mean a local Bayesian complexity with the critical radius $\varepsilon^\ast_n$.
Moreover, we will refer to a Bayesian risk bound like \eqref{eq:cor_renyi} as a Bayesian oracle inequality. 


\section{Model selection consistency and oracle inequalities for Bayesian procedures}\label{sec:main}
In this section, we start with our main result on the Bayesian model selection consistency for both the regular and fractional posterior distribution under some suitable identifiability condition on the truth. Next, we present Bayesian oracle inequalities for Bayesian model selection procedures that hold without the strong  identifiability condition.  Our Bayesian oracle inequality characterizes a trade-off between the approximation error and a local complexity of the model via the local Bayesian complexity, and illustrates the adaptive nature of averaging-based Bayesian procedures towards achieving an optimal rate of posterior convergence. Finally, we apply our theory to high dimensional nonparametric regression and density regression with variable selection.

\subsection{Bayesian model selection consistency}\label{Sec:BMSC}

Let us first introduce a few notations. Recall that $\lambdastar$ is the index corresponding to the smallest model $\MODEL_\lambda$ that contains $\Prob^{(n)}_{\thetastar}$. We use $\lambda > \lambda^\ast$ to mean all models $\MODEL_\lambda$ that are different from $\MODEL_\lambdastar$, but contain $\Prob^{(n)}_{\thetastar}$. According to our assumption~\eqref{Eqn:Model_assump} on the model space, we have
$\Theta_\lambdastar\subset \Theta_\lambda$ for each $\lambda > \lambda^\ast$.
Similarly, we use the notation $\lambda < \lambda^\ast$ to mean all models $\MODEL_\lambda$ that do not contain $\MODEL_{\lambdastar}$. Under this notation, the three subsets $\{\lambda:\, \lambda < \lambda^\ast\}$, $\{\lambdastar\}$ and $\{\lambda:\, \lambda > \lambdastar\}$ form a partition of the index set $\Lambda$. 
As a common practice, when $\widetilde{\Theta}$ is a set, we define $d_n(\theta,\widetilde{\Theta})$ as the infimum $d_n$-distance between $\theta$ and any point in $\widetilde{\Theta}$. $D^{(n)}_\alpha(\theta,\widetilde{\Theta})$ is defined in a similar way for the $\alpha$-divergence $D^{(n)}_\alpha$.   We make the following additional assumptions for Bayesian model selection consistency.

\paragraph{Assumption  B1 (Prior concentration for model selection):} 
\begin{enumerate}
\item(Parameter space prior concentration). There exists a sequence, denoted by $\{\uepsilon_{n,\lambdastar}:\,n\geq 1\}$, such that Assumption A2 holds with $\Theta=\Theta_\lambdastar$, $\Pi=\Pi_\lambdastar$ and $\varepsilon_n=\varepsilon_{n,\lambdastar}$. 

\item(Model space prior concentration). The prior weight $\pi_{\lambdastar}$ for $\MODEL_{\lambdastar}$ satisfies $\pi_{\lambdastar} \geq e^{-n\,\uepsilon_{n,\lambdastar}^2}$. 
\end{enumerate}

\

\noindent The parameter space prior concentration requires the prior distribution $\Pi_\lambdastar$ associated with the target model $\MODEL_\lambdastar$ to put enough mass into a KL neighbourhood around the truth $\thetastar$. However, we do not need any condition on the prior concentration for other $\Pi_\lambda$'s with $\lambda\neq\lambdastar$.
The model space prior concentration condition requires a lower bound on the prior mass assigned to the target model $\MODEL_\lambdastar$, which is in the same spirit as the first part on parameter space prior concentration. In fact, if $\pi_{\lambdastar} \leq e^{-C\, n\,\uepsilon_{n,\lambdastar}^2}$ for some sufficiently large constant $C>0$, then Theorem~\ref{Thm:rate} with $\mathcal{F}_n=\{\lambda:\, \lambda\neq \lambdastar\}$ implies that the marginal posterior probability of $\MODEL_\lambdastar$ tends to zero in probability as $n$ grows no matter how well it fits the data. 

\paragraph{Assumption B2 (Parameter space prior anti-concentration):} For each $\lambda>\lambdastar$, there exists a sequence, denoted by $\{\uepsilon_{n,\lambda}:\,n\geq 1\}$, such that $\varepsilon_{n,\lambda} \geq \varepsilon_{n,\lambdastar}$, and for the regular posterior distribution, $\Pi_\lambda (d_n(\theta,\, \thetastar) \leq M \uepsilon_{n,\lambda}) \leq e^{- H \, n \, \uepsilon_{n,\lambdastar}^2}$ holds for some sufficiently large constants $M>0$ and $H>0$;  for the fractional posterior distribution, $\Pi_\lambda (n^{-1}\,D^{(n)}_\alpha (\theta,\, \thetastar) \leq M \uepsilon_{n,\lambda}^2) \leq e^{- H \, n \, \uepsilon_{n,\lambdastar}^2}$.

\

\noindent This assumption ensures that the posterior probability of overly large models that contains the truth tends to zero as $n\to\infty$. The first part that $\uepsilon_{n,\lambdastar} \leq \uepsilon_{n, \lambda}$ for any $\lambda > \lambdastar$ requires the ``anti-contraction rate'' $\uepsilon_{n, \lambda}$ to honestly reflect the complexity of its associated model --- $\varepsilon_{n,\lambda}$ associated with $\MODEL_\lambda$ is expected to become slower as its parameter space $\Theta_\lambda$ becomes larger. We remark that the $\varepsilon_{n,\lambda}$ ($\lambda>\lambdastar$) defined in Assumption B2 can be identified with the rate $\lepsilon_{n}$ in the lower bound in Corollary~\ref{Corr:normpost} or Corollary~\ref{Cor::quasirate} with $\Theta=\Theta_\lambda$ and $\Pi=\Pi_\lambda$. This is  the reason we call it an ``anti-contraction rate''.
According to Corollary~\ref{Corr:normpost} or Corollary~\ref{Cor::quasirate},  Assumption B2 implies that the Bayes factor of $\MODEL_\lambda$ versus $\MODEL_\lambdastar$ 
\begin{align}\label{Eqn:BayesFactor}
\mbox{BF}_\alpha(\MODEL_\lambda;\,\MODEL_\lambdastar) :\,=\frac{\int_{\Theta_\lambda} \big(p^{(n)}_\theta (X^{(n)})\big)^\alpha\,\Pi_\lambda(d\theta)}{\int_{\Theta_\lambdastar} \big(p^{(n)}_\theta (X^{(n)})\big)^\alpha\,\Pi_\lambdastar(d\theta)},\qquad \mbox{for any $\lambda>\lambdastar$},
\end{align}
becomes exponentially small as $n\to\infty$, where $\alpha=1$ or $\alpha\in(0,1)$ depending on whether the regular or the fractional posterior distribution is used.  A combination of this fact with the model space prior concentration condition in A1 implies the posterior probability of $\{\lambda:\, \lambda>\lambdastar\}$, the set of all model indices corresponding to overly large models, tends to zero as $n\to\infty$.

\paragraph{Assumption B3 (Model identifiability):}
There exists $\delta_n>M \uepsilon_{n,\lambdastar}$, where $M$ is some sufficiently large constant independent of $n$, such that for all $\lambda<\lambdastar$, we have $d_n(\thetastar, \, \Theta_\lambda) \geq \delta_n$ for the regular posterior distribution, or $n^{-1}\, D^{(n)}_{\alpha}(\thetastar,\,\Theta_\lambda) \geq \delta_n^2$ for the fractional posterior distribution. 

\

\noindent The separation gap $\delta_n$ in Assumption B3 is the best approximation error of $\thetastar$ using elements in the parameter space $\bigcup_{\lambda<\lambdastar} \Theta_\lambda$ corresponding to such misspecified models. The lower bound condition $\delta_n \geq M \uepsilon_{n,\lambdastar}$ requires the best approximation error $\delta_n$ of misspecified models to be larger than the estimation error $\uepsilon_{n,\lambdastar}$ associated with $\MODEL_{\thetastar}$, so that the Bayes factor $\mbox{BF}_\alpha(\MODEL_\lambda;\,\MODEL_\lambdastar)$ defined in~\eqref{Eqn:BayesFactor} for $\lambda<\lambdastar$ also becomes exponentially small as $n\to\infty$, implying that $\MODEL_{\thetastar}$ is identifiable.
In the special case of high dimensional sparse linear regression, Assumption B3 becomes the  beta-min condition  \cite{Wainwright2009} on the minimal magnitude of nonzero regression coefficients, which is necessary for variable selection consistency.

\paragraph{Assumption B4 (Control on model complexity):} 
For each $\lambda< \lambdastar$, Assumptions A1 and A3 hold with $\Theta=\Theta_\lambda$, $\Pi=\Pi_\lambda$ and $\varepsilon_n=\delta_n$; and for each $\lambda> \lambdastar$, Assumptions A1 and A3 hold with $\Theta=\Theta_\lambda$, $\Pi=\Pi_\lambda$ and $\varepsilon_n=\varepsilon_{n,\lambda}$.
Moreover,  $\sum_{\lambda< \lambdastar} e^{-C\, n\,\delta_n^2} + \sum_{\lambda>\lambda} e^{-C\, n\,\varepsilon_{n,\lambda}^2} \leq 1$ holds for some sufficiently large constant $C>0$.

\

\noindent The first part of this assumption is the reminiscent of Assumptions A1 and A3 for controlling the model complexity of each $\MODEL_\lambda$ when the regular posterior distribution is used. The second part is due to technical reasons in the proof of Theorem \ref{Thm:Main} (some control when applying a union bound).

\btheos[Model selection consistency]
\label{Thm:Main}

\

\begin{enumerate}
\item(Regular posteriors).
Under Assumptions B1-B4, we have
\begin{align*}
&\Exs_{\thetastar}^{(n)}[ \Pi( \lambda = \lambdastar \, |\, X^{(n)}) ] \rightarrow 1,\quad \mbox{as $n\to\infty$.}
\end{align*}

\item(Fractional posteriors).
Under Assumptions B1-B3, we have
\begin{align*}
&\Exs_{\thetastar}^{(n)}[ \Pi_\alpha( \lambda = \lambdastar \, |\, X^{(n)}) ] \rightarrow 1,\quad \mbox{as $n\to\infty$.}
\end{align*}
\end{enumerate}
\etheos

\noindent Similar to the theorems in Section~\ref{Sec:ContractionRate} on the contraction rate of the regular and the fractional posterior distributions, the latter is more flexible than the former since it does not demand any assumption like B4 to control the model complexity in order to achieve the Bayesian model selection consistency. Comparing the anti-concentration conditions in B2 for the regular posterior and the fractional posterior, the former has an additional flexibility in the choice of the distance measure $d_n$ due to the additional test condition A1 made for the regular posterior. However, $d_n$ is typically a weaker or equivalent distance measure than the Renyi divergence with certain order $\alpha$. Therefore, the anti-concentration condition for the regular posterior is often a stronger condition than that for the fractional posterior. For the fractional posterior distribution, when a $D_\alpha^{(n)}$ neighborhood of $\thetastar$ is comparable to the KL-neighborhood in the definition of local Bayesian complexity (for example, regression problems with Gaussian errors), a combination of Assumptions B1 and B2 imposes a two-sided constraint on the local Bayesian complexity on the true model $\MODEL_{\lambda^\ast}$, implying that the fractional posterior tends to concentrate all its mass on the model that contains the truth $\thetastar$ with smallest local Bayesian complexity.

\subsection{Posterior contraction rate and Bayesian oracle inequality for model selection}\label{Sec:BOI}
Following Section~\ref{Sec:PACBayes}, we show in this subsection that even when model selection consistency does not hold, a Bayesian procedure may still achieve estimation optimality by adaptively concentrating on models with fastest rate of contraction.

\paragraph{Regular posterior distributions:}
First, we focus on the use of the regular posterior distribution~\eqref{Eqn:normalpost}. We require a stronger version of the sieve sequence condition.

\paragraph{Assumption A3$'$ (Sieve sequence condition):} 
For some constant $D>1$, and any $\varepsilon >  0$, there exists a sequence of sieves $\mathcal{F}_{n} \subset \Theta$, $n=1,2, \ldots$, such that 
\begin{align*}
\log N(\varepsilon, \, \mathcal{F}_{n}, \, d_n) \leq n\,\varepsilon^2
\qquad\mbox{and}\qquad \Pi_\lambda(\mathcal{F}_{n}^c) \leq e^{-D\,n\,\varepsilon^2}.
\end{align*}

\noindent Assumption A3$'$ is stronger than Assumption A3 since it assumes the existence of the sieve sequence for any $\varepsilon>0$ rather than a single $\varepsilon=\varepsilon_{n,\lambda}$. In most cases, the sieve sequence constructed for verifying Assumption A3 naturally extends to general $\varepsilon$ (see the examples in the following subsections). The following theorem shows that when Assumptions A1 and A3$'$ are true with $\Theta=\Theta_\lambda$, $\Pi=\Pi_\lambda$ for each $\lambda\in\Lambda$, the posterior distribution automatically adapts to the best contraction rate over all models.

\btheos[Contraction rate of regular posterior distributions]\label{Thm:normalpostMSrate}
If Assumptions A1 and A3$'$ with $\Theta=\Theta_\lambda$, $\Pi=\Pi_\lambda$ hold for each $\lambda\in\Lambda$, then as $n\to\infty$,
\begin{align*}
\Exs_{\thetastar}^{(n)}\bigg[ \Pi\bigg( d^2_n(\theta, \,\thetastar) 
&> M \,\bigg(\min_{\lambda\in\Lambda} \, \Big\{ \min_{\varepsilon_\lambda>0}\Big\{\varepsilon^2_{\lambda} 
+ \Big(-\frac{1}{n}\log \Pi_\lambda(B_n(\theta^\ast,\,\varepsilon_\lambda))\Big)\Big\} \\
&\qquad\qquad\quad \ \ + \Big(-\frac{1}{n}\log \pi_\lambda\Big)\Big\} + \frac{\log |\Lambda|}{n}\bigg) \, \bigg|\, X^{(n)}\bigg) \bigg] \rightarrow 0.
\end{align*}
\etheos

\noindent Theorem~\ref{Thm:normalpostMSrate} shows that the overall rate of contraction under model selection is composed of three terms. The first term, 
$$
\varepsilon_{n,\lambda}^2:\,=\min_{\varepsilon_\lambda>0}\Big\{\varepsilon^2_{\lambda} + \Big(-\frac{1}{n}\log \Pi_\lambda(B_n(\theta^\ast,\,\varepsilon_\lambda))\Big)\Big\}, \quad \mbox{for $\lambda\in\Lambda$},
$$
characterizes the rate of posterior contraction under model $\MODEL_\lambda$ thorough balancing between the approximation risk $\varepsilon^2_{\lambda}$ and local Bayesian complexity $-(1/n)\log \Pi_\lambda(B_n(\theta^\ast,\,\varepsilon_\lambda))$. The second term $-(1/n)\log \pi_\lambda$ can be viewed as the local Bayesian complexity 
over the space of model index set $\Lambda$ that reflects the prior belief over different models. The third term $\sqrt{\log|\Lambda|/n}$ is a model selection uncertainty term, proportional to the logarithm of the cardinality of models, allowing us to attain the estimation consistency with $|\Lambda|$ growing at most exponentially large in the sample size. Under the special case of an ``objective prior" that assigns equal prior weight to each model, the second model space complexity term has the same order as the third model selection uncertainty term. In general, the posterior distribution adaptively achieves the fastest rate of contraction that optimally trades-off those three terms, that is,
\begin{align*}
\min_{\lambda\in\Lambda} \Big\{\varepsilon_{n,\lambda}^2 + \Big(-\frac{1}{n}\log \pi_\lambda\Big)\Big\} + \frac{\log |\Lambda|}{n}.
\end{align*}
Theorem~\ref{Thm:normalpostMSrate} requires much weaker conditions than Theorem~\ref{Thm:Main}. The former does not require any model identifiability condition nor the anti-concentration condition, which is consistent with the intuition that an optimal convergence rate can be attainable without model selection consistency.

\paragraph{Fractional posterior distributions:}
Now let us turn to the fractional posterior distribution~\eqref{Eqn:quasilpost} based on fractional likelihood functions. First, we present our result on its contraction rate.

\btheos[Contraction rate of fractional posterior distributions]\label{Thm:quasipostMSrate}
For any $\alpha\in(0,1)$, we have, as $n\to\infty$,
\begin{align*}
\Exs_{\thetastar}^{(n)}\bigg[ \Pi_\alpha\bigg( \frac{1}{n} D^{(n)}_\alpha(\theta, \,\thetastar) 
 >&\, \frac{M}{1-\alpha} \,\min_{\lambda\in\Lambda} \bigg\{ \min_{\varepsilon_\lambda>0}\Big\{\varepsilon^2_{\lambda} + \Big(-\frac{1}{n}\log \Pi_\lambda(B_n(\theta^\ast,\,\varepsilon_\lambda))\Big)\Big\}\\
&\qquad\qquad\qquad+ \Big(-\frac{1}{n}\log \pi_\lambda\Big)\bigg\} \, \bigg|\, X^{(n)}\bigg) \bigg] \rightarrow 0.
\end{align*}
\etheos

\noindent Comparing to Theorem~\ref{Thm:normalpostMSrate}, Theorem~\ref{Thm:quasipostMSrate} does not demand any conditions like Assumption A1 or A3 for controlling the complexity of models, and obviates the model selection uncertainty term $\log|\Lambda|/n$ (which is due to a union bound for controlling the covering entropy of $\Theta=\bigcup_{\lambda\in\Lambda}\Theta_\lambda$). As a consequence, a fractional posterior distribution may lead to faster posterior contraction rate than a regular posterior distribution (see our example in Section~\ref{Sec:NonReg} for variable selection of nonparametric regression). More interestingly, a fractional posterior distribution may still be able to achieve estimation consistency without the typical constraint on the number of models, such as  $\log |\Lambda| = o(n)$ provided the prior $\{\pi_\lambda\}$ is properly chosen (more discussion is provided after Theorem \ref{Thm:quasipostMS_Oracle} below). This interesting phenomenon can be explained by the averaging-based analysis nature of Bayesian approaches, as opposed to the worst case analysis nature of optimization based approaches (more discussions on the comparison between typical Bayesian analysis and frequentist analysis can be found in \cite{bhattacharya2016bayesian}).  By applying Theorem~\ref{Thm:PAC-Bayes-one-model}, we have the following Bayesian oracle inequality for the fractional posterior distribution. 
\btheos[Bayesian oracle inequality for model selection]\label{Thm:quasipostMS_Oracle}
For any $\alpha\in(0,1)$, if $M$ is some sufficiently large constant independent of $\alpha$ and $n$, then with $\Prob^{(n)}_{\theta^\ast}$ probability tending to one as $n\to\infty$,
\begin{align*}
 &\int \frac{1}{n}D^{(n)}_\alpha(\theta, \thetastar) \Pi_\alpha(d\theta\,|\, X^{(n)}) \le\\
 &\qquad \frac{M}{1 - \alpha} \min_{\lambda\in\Lambda} \, \Big\{ \min_{\varepsilon_\lambda>0}\Big\{ \varepsilon^2_{\lambda} + \Big(-\frac{1}{n}\log \Pi_\lambda(B_n(\theta^\ast,\,\varepsilon_\lambda))\Big)\Big\} + \Big(-\frac{1}{n}\log \pi_\lambda\Big)\Big\}.
\end{align*}
\etheos

\noindent Similar to the observations in Theorem~\ref{Thm:quasipostMSrate} and the remarks after Theorem~\ref{Thm:Main}, this theorem shows that the Bayesian risk for model selection is the best trade-off between the approximation error $\varepsilon_{n,\lambda}^2$, where $\varepsilon_{n,\lambda}$ is defined as the minimizer of $\varepsilon_\lambda$ in the inner minimization step for a fixed $\lambda$, and a model complexity penalty term consisting of the local Bayesian complexity $-(1/n) \log \Pi(B_n(\thetastar,\varepsilon_{n,\lambda}))$ of the model $\MODEL_\lambda$ and the local Bayesian complexity  $ -(1/n) \log \pi_{\lambda}$ of the model space $\MODEL$. Due to the averaging-based analysis nature of a Bayesian procedure, the proof does not rely on any sophisticated empirical process tools, but only elementary Chebyshev's inequality.
Again, the overall Bayesian risk does not explicitly depend on the cardinality $|\Lambda|$ of the model space. However, when no informative prior knowledge is available on which models are more likely, we would like to put a non-informative prior distribution on the model space, which implicitly entails some constraints on $|\Lambda|$.
For example, in order for $-(1/n)  \log \pi_{\lambda}$ to not substantially exceed $-(1/n) \log \Pi(B_n(\thetastar,\varepsilon_{n,\lambda}))$ for any $\lambda\in\Lambda$, we need a condition like $\pi_\lambda \geq e^{-n\,\varepsilon_{n,\lambda}^2}$, $\forall \lambda\in\Lambda$. This condition implies a constraint on $|\Lambda|$ through 
$$\sum_{\lambda\in\Lambda} e^{- n\, \varepsilon_{n,\lambda}^2} \leq \sum_{\lambda\in\Lambda}\pi_\lambda =1,$$
implying that $\log |\Lambda| = o(n)$ when $\sup_{\lambda\in\Lambda} \varepsilon_{n,\lambda}=O(n^{-1/2})$.

\subsection{Application to high dimensional nonparametric regression}\label{Sec:NonReg}
In this subsection, we consider high dimensional nonparametric regression 
\begin{align*}
Y_i = f(X^i) + w_i,\qquad w_i\sim N(0,\sigma^2),
\end{align*}
 where $X^i\in[0, 1]^p$ is the $i$th observed $p$-dimensional covariate vector, $Y_i$ is the response, $w_i$'s are i.i.d.~Gaussian noise and $f$ is an unknown regression function. For simplicity, we assume that the error standard deviation $\sigma$ is known. In the high dimensional regime, $p$ is allowed to be comparable or much larger than the sample size $n$. 
 \cite{zou2010nonparametric,savitsky2011variable} developed  Bayesian methods for variable selection in this context by placing a sparsity inducing Gaussian process prior on $f$.  The goal of this section is to theoretically investigate variable selection consistency of similar Bayesian procedures.  

 In the fixed-dimensional context, 
  \cite{bhattacharya2014anisotropic} developed theory for optimal posterior contraction using Gaussian process priors with multiple-bandwidths under the assumption that the true regression function depends on a subset of predictors.    \cite{yang2015b} studied the minimax estimation risk of high dimensional nonparametric regression under three different structural assumptions on $f$. In this article,  we will focus on the sparsity structure with respect to the dependence on covariates as in   \cite{bhattacharya2014anisotropic}. The sparsity constraint says that $f$ only depends on a small subset $I^\ast\subset [p]:\,=\{1,2,\ldots,p\}$ of size $d^\ast$. The model space $\Lambda$ is the power set of $[p]$. Therefore, we will identify each model index $\lambda$ with a subset $I$ of $[p]$. Without loss of generality, we may assume that the true model corresponds to the index set $I^\ast = \{1,\ldots, d^\ast\}$ by reordering the covariates, where $d^\ast$ is the number of important covariates included in the true model $\MODEL_{I^\ast}$:
\begin{align*}
Y_i = f(X^i_{I^\ast}) + w_i,\qquad w_i\sim N(0,\sigma^2),
\end{align*} 
 where for any $I \subset [p]= \{1, 2, \ldots, p\}$ and any vector $b\in \bbR^p$, we use the notation $b_I\in \bbR^I$ to denote the projection of $b$ onto the coordinates along $I$.  
Under this setup, the sets $\{\lambda:\,\lambda < \lambdastar\}$ and $\{\lambda:\,\lambda > \lambdastar\}$ become $\{I\in[p]:\, I \not\supset I^\ast\}$ and $\{I\in[p]:\, I \supsetneq I^\ast\}$ respectively. For technical reasons, we assume that an upper bound $A_{\infty}$ on the sup-norm $\|f\|_{\sup}$ is known.
We made the following assumption on the distribution of the $p$-dimensional covariate vector $X=(X_1,\ldots,X_p)$, under which Assumption B3 has a simpler equivalent formulation.

\paragraph{Assumption NP-C (Independent design):} $X_1,\ldots,X_p$ are independent and identically distributed with marginal density $\mu$ over the unit interval $[0,1]$, where $\mu$ satisfies $\mu(x)\geq c_0>0$, $\forall x\in[0,1]$.

\

Let the parameter space associated with $\MODEL_I$ be $\Theta_I = \mathbb{C}^\alpha([0,1]^I)$, the H\"older class of functions on 
$\mathbb{R}^I$ having smoothness $\alpha$ and choose the distance metric $d_n(f_1, f_2) =  \norm{f_1 - f_2}_{\mu, 2}:\,= \int_{[0,1]^p} \{f_1(x) -f_2(x)\}^2 \mu^p(x) dx$ for $f_1, f_2 \in \mathbb{L}_2(\mu^p; \, [0,1]^{p})$ where $\mu^p$ is a product probability measure on $[0,1]^{p}$ and $\mathbb{L}_2(\rho^n;\,  [0,1]^{p}) = \{f: [0,1]^{p}\to \mathbb{R}: \int f^2(x)\, \mu^p(x)\,dx < \infty\}$. 
 Let $\mathbb{N}_0$ be the set of all non-negative integers and let $\{e_v(x) = \prod_{j\in I}e_{v_j}(x_j): v \in \mathbb{N}_0^{I^*}\}$ be a set of orthonormal basis functions of  $\mathbb{L}_2(\mu; [0, 1]^{I^*})$ with $e_0 \equiv 1$. Under random design and the assumption that the regression function is bounded,  we have $D_{\alpha} (f_1,f_2) = C\, \alpha\, \|f_1-f_2\|_{\mu,2}^2= C\, \alpha\, d_n^2(f_1,\,f_2)$, where $C$ is some constant only depending on $A_{\infty}$.
Therefore, Assumption B3 becomes:

\paragraph{Assumption NP-B3 (Identifiability condition for nonparametric regression):} There exists $\delta_n > 0$ such that for all $I \not\supset I^\ast$, we have $d_n(f^*, \Theta_I) \geq \delta_n \geq M \varepsilon_{n, I^\ast}$ under the regular posterior distribution; and $d_n(f^*, \Theta_I) \geq C\, \alpha^{-1/2}\, \delta_n$ $\geq C\, \alpha^{-1/2}\,M \varepsilon_{n, I^\ast}$ under the $\alpha$-fractional posterior distribution.

\

The following lemma provides a characterization of gap $\delta_n$ when using a wrong model that miss at least one important covariate.

\bprops[Identifiability for variable selection in nonparametric regression]
\label{Prop:NPB2}
Under Assumption NP-C, we can choose
\begin{align*}
\delta_n^2:\,&=\inf_{I \not\supset I^\ast} d_n^2(f^\ast, \Theta_I) \\
&= \min_{j=1,\ldots,d^\ast}
\bbE[ \mbox{Var}[f^\ast(X_{I^\ast})\, | \, X_{I^\ast\setminus\{j\}}]]
= \min_{j=1,\ldots,d^\ast}\sum_{v \in \mathbb{N}_0^{I^\ast}, v_{j} \neq 0} \langle f^*, e_v\rangle^2.
\end{align*}
\eprops

The conclusion of this proposition is intuitive: $\delta_n$ reflects the average relative change in one coordinate when the rest are held fixed. Moreover, due to the independence assumption on the design, $\delta_n$ only depends on the important covariates $X_1,\ldots,X_{d^\ast}$.
The last expression in the proposition is of a similar form as the assumption in Equation (2) in \cite{comminges2012}. Intuitively, this proposition, or Assumption NP-B3, states that the basis coefficients of $f^*$ associated with the relevant variables should be sufficiently large for statistical identifiability.

\

Now we specify our Bayesian variable selection model for high dimensional nonparametric regression under the sparsity constraint. For simplicity, we identify the subset $I$ with a $p$-dimensional binary vector, called the inclusion indicator vector, whose $j$th element is one if and only if $j$ is in the subset. Under this identification, we can write $I=(I_1,\ldots,I_p) \in\{0,1\}^p$. Conditional on a  $\MODEL_I$ being selected, we place a Gaussian process (GP) prior with inverse Gamma bandwidth \cite{vandervaart2009} on $f$. The law of a $d$-dimensional GP  $W=(W_x;x\in [0,1]^d)$ is completely determined by its mean function $m(x)=EW_x$ and covariance function $K(x,x')=E(W_x-m(x))(W_{x'}-m(x'))$. \cite{vandervaart2009} recommended using a zero mean GP with a squared exponential kernel $K^a(x,x'):\,=\exp(-a^2||x-x'||^2)$, where $a$ is an inverse bandwidth hyperparameter whose prior satisfies $A^d \sim \mbox{Ga}(a_1,a_2)$, where $\mbox{Ga}(a,b)$ stands for the Gamma distribution with shape and scale parameter $(a,b)$ and $(a_1,a_2)$ are some constants. \cite{vandervaart2009} showed that in the absence of model selection, a Bayesian regression model with this GP prior leads to a minimax-optimal posterior contraction rate up to $\log n$ factors. For technical reasons, we consider a variant of this prior, denoted by $\Pi^{\mathrm{GP}}$,  as the conditional GP prior on the set $\{ \|f\|_{\sup} \leq A_\infty, A \geq n^{\frac{1}{2\beta +d}}\}$.
Some discussions on conditioning on $\{ \|f\|_{\sup} \leq A_\infty\}$ for random design nonparametric regression can be found in \cite{yang2015b}.  We impose the constraint $\{A \geq n^{\frac{1}{2\beta +d}}\}$ in order to verify the prior anti-concentration assumption B2. This specific form of prior is only designed to illustrate the applicability of our general theory on Bayesian model selection consistency, and we do not pursue a more practically convenient prior specification in the current paper.
In our model selection setup, we use the notation $\Pi^{\mathrm{GP}}_{I}$ to denote the $\Pi^{\mathrm{GP}}$ prior over functions of $X_I \in [0,1]^{I}$.

Following \cite{yang2015b}, we consider the following GP variable selection (GPVS) prior for high dimensional nonparametric regression with variable selection:
\begin{equation}\label{eq:GPVSp}
\begin{aligned}
\pi_I &\propto p^{-|I|}(1-p^{-1})^{p-|I|}\mathbb{I}(|I|\leq d_0),\\
f\, |\, I &\sim \Pi^{\mathrm{GP}}_I,
\end{aligned}
\end{equation}
where $\mathbb{I}(\cdot)$ is the indicator function, $d_0$ is a prespecified hyperparameter, interpreted as the prior belief on the maximum number of important predictors.  The following theorem shows that a Bayesian procedure under this hierarchical prior is Bayesian model selection consistent.  In the applied Bayesian literature, similar priors have been used to select variables in the context of gene-expression studies \cite{zou2010nonparametric,savitsky2011variable}. 

\btheos[Variable selection consistency for high dimensional nonparametric regression]\label{Thm:HNRconsistency}
Assume that the true regression function $f^\ast \in \mathbb{C}^\beta([0,1]^{I^\ast})$ only depends on $x_{I^\ast}$.  Let $d^\ast = |I^\ast|$.
If $\beta \geq d_0/2$, and 
\begin{enumerate}
\item  for the regular posterior distribution, Assumption NP-B3 holds with
\begin{align*}
\delta_n=M\, \bigg(n^{-\frac{\beta}{2\beta + d^\ast}} \, (\log n)^{\frac{4\beta +d^\ast}{4\beta +2 d^\ast}} \wedge \sqrt{\frac{d_0\log p}{n}}\bigg),
\end{align*}
for sufficiently large $M$;

\item for the $\alpha$-fractional posterior distribution, Assumption NP-B3 holds with
\begin{align*}
\delta_n=\frac{M}{\sqrt{\alpha\,(1-\alpha)}}\,\bigg( n^{-\frac{\beta}{2\beta + d^\ast}} \, (\log n)^{\frac{4\beta +d^\ast}{4\beta +2 d^\ast}} \wedge \sqrt{\frac{d^\ast \log p}{n}}\bigg),
\end{align*}
for sufficiently large $M$ independent of $\alpha$,
\end{enumerate}
 then under the GPVS prior~\eqref{eq:GPVSp}, it holds that 
\begin{align*}
&\Exs_{f^\ast}^{(n)}[ \Pi_\alpha( I = I^\ast \, |\, X_1,\ldots,X_n) ] \rightarrow 1,\quad \mbox{as $n\to\infty$.}
\end{align*}
\etheos

The main difficulty in the proof is the verification of the prior anti-concentration condition B2. To achieve that, we propose a general technique for obtaining lower bounds to small ball probabilities of any stationary GP. First,  the small ball probability lower bound is related to a lower bound of the covering entropy of the Reproducing kernel Hilbert space associated with the GP. Second, we use a volume argument to obtain a lower bound of the covering entropy, where the lower bound depends on the eigenvalues associated with the covariance kernel of the GP. Finally, we provide a general lemma (Lemma~\ref{Eqn:EigenSystem}) to characterize the eigenvalues of any one-dimensional stationary kernel over $[0,1]^2$, which may be of independent interest. Comparing the conditions using the regular and the fractional posterior distribution in Theorem~\ref{Thm:HNRconsistency}, the only difference is on the lower bound of the gap $\delta_n$. For the regular posterior distribution, the lower bound of  $\delta_n$ depends on the largest dimension $d_0$ of the model in the support of the prior.  For the fractional posterior distribution, the lower bound depends on the potentially much smaller dimension $d^\ast$ of the true regression function, although sacrificing a multiplicative factor $\alpha^{-1/2}\,(1-\alpha)^{-1/2}$. The next result shows the rate of contraction for high dimensional nonparametric regression under the GPVS prior.  The proof of Theorem  \ref{Thm:HNR_rate} is a straightforward application of Theorems \ref{Thm:normalpostMSrate} and \ref{Thm:quasipostMSrate} and hence is omitted.

\btheos[Contraction rate for high dimensional nonparametric regression]\label{Thm:HNR_rate}
Assume that the true regression function $f^\ast$ only depends on $x_{I^\ast}$ and is $\beta$-H\"{o}lder smooth. Let $d^\ast = |I^\ast|$.
Under the GPVS prior~\eqref{eq:GPVSp}:
\begin{enumerate}
\item the regular posterior distribution has a contraction rate relative to the $\|\cdot\|_{\mu,2}$ norm at least
\begin{align*}
\varepsilon_n=n^{-\frac{\beta}{2\beta + d^\ast}} \, (\log n)^{\frac{4\beta +d^\ast}{4\beta +2 d^\ast}} \wedge \sqrt{\frac{d_0 \log p}{n}};
\end{align*}

\item the $\alpha$-fractional posterior distribution has a contraction rate relative to the $\|\cdot\|_{\mu,2}$ norm at least
\begin{align*}
\varepsilon_n=\frac{1}{\sqrt{\alpha\,(1-\alpha)}}\, \bigg(n^{-\frac{\beta}{2\beta + d^\ast}} \, (\log n)^{\frac{4\beta +d^\ast}{4\beta +2 d^\ast}} \wedge \sqrt{\frac{d^\ast \log p}{n}}\bigg).
\end{align*}
\end{enumerate}
\etheos

\noindent Similar to Theorem~\ref{Thm:HNRconsistency}, fractional posteriors reveal better adaptation on the sparsity level $d^\ast$ in the contraction rate than regular posteriors. It is unclear whether this difference in the rate is due to the inefficiency of regular posteriors or deficiency of current proof techniques (there are examples where fractional posteriors contract while the regular posterior does not, for example, see \cite{bhattacharya2016bayesian}).


\subsection{Application to density regression with variable selection}\label{Sec:DenReg}
In this section, we study consistency of  variable selection in the context of estimating conditional densities which may possibly
depend on a fixed subset of predictors. To that end, we define $\mathcal{Y} \subset \mathbb{R}$ be the response space, $\mathcal{X} = [0,1]^p$ be the covariate space, and $\mathcal{Z}=\mathcal{Y} \times \mathcal{X}$. Let $\mathcal{F}$ denote a space of conditional densities with respect to the Lebesgue measure,
\[
\mathcal{F} = \bigg\{ f:\mathcal{Y} \times \mathcal{X} \rightarrow [0, \infty) \mbox{ - Borel measurable, } 
\int f(y|x)dy =1, \; \forall x \in \mathcal{X} \bigg\}.
\]
Suppose $\{(Y^i,X^i), i=1, \ldots, n\}$ is a random sample from the joint density $f_0 \mu^p$, where  $f_0 \in \mathcal{F}$ and
$\mu^p$ is a product probability measure in $\mathcal{X}$ with respect to the Lebesgue measure. 
Let $P_0$ and $E_0$ denote the probability measure and expectation corresponding to $f_0 \mu^p$.
For $f_1, f_2 \in \mathcal{F}$, 
\[d_h(f_1,f_2) = \left(\int \left(\sqrt{f_1(y|x)}-\sqrt{f_2(y|x)}\right)^2 \mu^p(x) dy dx \right)^{1/2}\;\mbox{ and }\]  
	\[d_1(f_1,f_2) = \int  |f_1(y|x)-f_2(y|x)| \mu^p(x) dy dx\]
denote analogs of the Hellinger and total variation distances correspondingly.  Also, let us denote the Hellinger distance for the joint densities by $d_H$. 

Let us denote the largest integer that is strictly smaller than $\beta$ by 
$\lfloor \beta \rfloor$.
For $L:\mathcal{Z}\rightarrow[0,\infty)$, $\tau_0 \geq 0$, and $\beta>0$,
a class of locally H\"older functions, $\mathcal{C}^{\beta,L,\tau_0}(\mathcal{Z})$, consists of 
$f:\mathbb{R}^d\rightarrow \mathbb{R}$ such that for $k=(k_1,\ldots,k_d)$, $k_1+\cdots+k_d \leq \lfloor \beta \rfloor$, mixed partial derivative of order $k$, $D^k f$, is finite and for $k_1+\cdots+k_d = \lfloor \beta \rfloor$ and $\Delta z \in \mathcal{Z}$,
\[
|D^k f (z+\Delta z) - D^k f (z)| \leq L(z) ||\Delta z||^{\beta-\lfloor \beta \rfloor} e^{\tau_0 ||\Delta z||^2}.  
\]
We identify $\Theta_I$ as $\mathcal{C}^{\beta,L,\tau_0}(\mathcal{Y} \times [0, 1]^I)$.  
The following identifiability  condition is the analogue of NP-B3 in the case of density regression.  
In this case, for $f_1, f_2 \in   \mathcal{C}^{\beta,L,\tau_0}(\mathcal{Z})$, $D_{\alpha} (f_1,f_2)$ is given by
\begin{eqnarray*}
D_\alpha(f_1, f_2) =   \frac{1}{\alpha -1} \log \int_{\mathcal{Z}} f_1^\alpha(y|x)  f_2^{1-\alpha}(y|x) \mu^p(x) dydx. 
\end{eqnarray*}

\paragraph{Assumption DR-B3 (Identifiability condition for density regression):}
There exists $\delta_n > 0$ such that for all $I \not\supset I^\ast$, we have $d_h(f^*, \Theta_I) \geq \delta_n \geq M \varepsilon_{n, I^\ast}$ under the regular posterior distribution; and $D_{\alpha}(f^*, \Theta_I) \geq  \delta_n^2 \geq  M^2 \varepsilon^2_{n, I^\ast}$ under the $\alpha$-fractional posterior distribution.

In the special case when $\alpha \in [1/2,1)$, it is sufficient to have $d_{h}(f^*, \Theta_I) \geq  (1/2)\delta_n^2$ in DR-B3 for the fractional posterior distribution.  
\paragraph{Assumption DGP (Assumption on the data generating process):}
We assume that $f^* \in \mathcal{C}^{\beta,L,\tau_0}$ and depends on  $x_{I^\ast}$ with $d^\ast = |I^\ast|$,  $\mu$ is assumed to be bounded from below and above and  
 for all $k \leq \lfloor \beta \rfloor$ and some $\varepsilon>0$,
	\begin{equation}
	\label{eq:asnE0Dff0_Lf0}
	\begin{aligned}
	\int_{\mathcal{Z}} \left|\frac{D^k f^*(y|x)}{f^*(y|x)}\right|^{(2\beta+\varepsilon)/k} f^*(y|x) \, dy\, dx&< \infty,
	\\
	\int_{\mathcal{Z}} \left|\frac{L(y,x)}{f^*(y|x)}\right|^{(2\beta+\varepsilon)/\beta} f^*(y|x) \, dy\, dx&< \infty.
\end{aligned}
	\end{equation} 
Assume, for all $x \in [0, 1]^{I^\ast}$, all sufficiently large $y \in \mathcal{Y}$ and some positive $(c,b,\tau)$, 
\begin{equation}
	\label{eq:asnf0_exp_tails}
f^*(y|x) \leq c \exp(-b ||y||^\tau). 
	\end{equation}

\

In our model selection setup, we use the notation $\Pi^{\mathrm{DR}}_{I}$ to denote a prior over conditional densities defined on  $\mathcal{Y} \times [0,1]^{I}$. 
$\Pi^{\mathrm{DR}}_{I}$ is defined by a location mixture of normal densities 
\begin{equation}
\label{eq:cond_mix_def}
p(y|x, \theta, m) = \sum_{j=1}^m\frac{ \alpha_j \exp\{-0.5||x-\mu_j^x||^2/\sigma^2 \}   
}
{\sum_{i=1}^m \alpha_i \exp\{-0.5||x-\mu_i^x||^2/\sigma^2 \}   
}\phi_{\mu_j^y,\sigma}(y),
\end{equation}
where $\sigma \in (0,\infty)$ and we choose
a prior on $\theta=(\mu_j^y, \mu_j^x, \alpha_j, j=1,2,\ldots,m)$,
where $\alpha_j \in [0,1]$, $\mu_j^y \in \mathbb{R}^{}$, $\mu_j^x \in \mathbb{R}^{|I|}$,
We assume the following conditions on the prior.  
Prior for 
$(\alpha_1,\ldots,\alpha_m)$ given $m$ is Dirichlet$(a/m,\ldots,a/m)$, $a > 0$ restricted to the region  $\{\alpha_j > b/m, j=1, \ldots, m\}$ where 
$b > 1$ a fixed constant.  
A priori, $\mu_{j}=(\mu_{j}^y,\mu_{j}^x)$'s are independent from other parameters and across $j$,
and $\mu_{j}^y$ is independent of  $\mu_{j}^x$.
Prior density for $\mu_{j}^x$ is bounded away from 0 on $\mathcal{X}$ and equal to 0 elsewhere.	
Prior density for $\mu_{j}^y$ is bounded below for some $a_1, a_{2}, \tau_1 > 0$ by 
\begin{equation}
	\label{eq:asnPrior_mu_lb}
	a_{1}\exp(-a_{2} ||\mu_j^y||^{\tau_1} ),
\end{equation}
	and for some $a_{3}, \tau_2>0$ and all sufficiently large $r > 0$,
	\begin{equation}
	\label{eq:asnPrior_mu_tail_ub}
	1- \Pi(\mu_{j}^y \in [-r,r]^{}) \leq \exp(-a_{3} r^{\tau_2}). 
\end{equation}
Define 
\begin{eqnarray}\label{eq:modelrate}
\varepsilon_{n,I} = n^{-\frac{\beta}{2\beta + |I| +1}} \, (\log n)^{t}
\end{eqnarray}
with 
$t > t_0 +  \max \{0, (1- \tau_1)/2\}$, 
$t_0=  \{(|I| +1)s + \max\{\tau_1,1,\tau_2/\tau\}) / \{2 + (|I|  +1)/\beta\}$, 
and  $s = 1 + 1/\beta + 1/\tau$. Set $\sigma=\sigma_n = [\varepsilon_{n, I} / \log (1/ \varepsilon_{n, I}) ]^{1/\beta}$ and $m =  \sigma_n^{-|I|} \{\log (1/ \varepsilon_{n, I}) \}^{|I| +|I|/\tau}$.  Extensions to the cases where $\sigma$ and $m$ are assigned prior distributions are not explored in the current paper.  
Similar to  \eqref{eq:GPVSp} in the nonparametric regression case, we assume the DRVS prior with 
\begin{equation}\label{eq:DRVSp}
\begin{aligned}
\pi_I &\propto p^{-|I|}(1-p^{-1})^{p-|I|}\mathbb{I}(|I|\leq d_0) \\
&p(\cdot |\cdot, \theta, m) \, |\, I \sim \Pi^{\mathrm{DR}}_I.
\end{aligned}
\end{equation}
The following theorem shows that a Bayesian procedure under this hierarchical prior is Bayesian model selection consistent.
\btheos[Variable selection consistency for high dimensional density regression]\label{Thm:HDRconsistency}
Assume that the true conditional density $f^*$ satisfies assumption DGP with $\beta > d_0$  and
\begin{enumerate}
\item  for the regular posterior distribution, Assumption DR-B3 holds with
\begin{align*}
\delta_n=M\, \bigg(n^{-\frac{\beta}{2\beta + d^\ast +1}} \, (\log n)^t\wedge \sqrt{\frac{d_0\log p}{n}}\bigg),
\end{align*}
for sufficiently large $M$;

\item for the $\alpha$-fractional posterior distribution, Assumption DR-B3 holds with
\begin{align*}
\delta_n=\frac{M}{\sqrt{\alpha\,(1-\alpha)}}\,\bigg( n^{-\frac{\beta}{2\beta + d^\ast +1}} \, (\log n)^t \wedge \sqrt{\frac{d^\ast \log p}{n}}\bigg),
\end{align*}
\end{enumerate}
with  $t > t_0 +  \max \{0, (1- \tau_1)/2\}$, 
$t_0=  \{(d^\ast +1)s + \max\{\tau_1,1,\tau_2/\tau\}) / \{2 + (d^\ast  +1)/\beta\}$, 
and  $s = 1 + 1/\beta + 1/\tau$.  Then for sufficiently large $M$ independent of $\alpha$,
 under the DRVS prior~\eqref{eq:DRVSp}, it holds that 
\begin{align*}
&\Exs_{f^\ast}^{(n)}[ \Pi_\alpha( I = I^\ast \, |\, X_1,\ldots,X_n) ] \rightarrow 1,\quad \mbox{as $n\to\infty$.}
\end{align*}
\etheos

Since Assumption A3$'$ is satisfied due to Theorem 4.1 of \cite{norets2016}, we derive a rate of contraction theorem under the DRVS prior similar to 
Theorem \ref{Thm:HNR_rate}.  The proof is straightforward from \cite{norets2016} and Theorems \ref{Thm:normalpostMSrate} and \ref{Thm:quasipostMSrate} and hence is omitted.  
\btheos[Contraction rate for high dimensional density regression]\label{Thm:HDR_rate}
Assume that the true conditional density $f^*$ satisfies assumption DGP.   Then under the DRVS prior~\eqref{eq:DRVSp}:
\begin{enumerate}
\item the regular posterior distribution has a contraction rate relative to $d_h$ at least
\begin{align*}
\varepsilon_n=n^{-\frac{\beta}{2\beta + d^\ast+1}} \, (\log n)^t \wedge \sqrt{\frac{d_0 \log p}{n}};
\end{align*}

\item the $\alpha$-fractional posterior distribution has a contraction rate relative to $d_h$ at least
\begin{align*}
\varepsilon_n=\frac{1}{\sqrt{\alpha\,(1-\alpha)}}\, \bigg(n^{-\frac{\beta}{2\beta + d^\ast+1}} \, (\log n)^t \wedge \sqrt{\frac{d^\ast \log p}{n}}\bigg).
\end{align*}
\end{enumerate}
where $t$ is defined in Theorem \ref{Thm:HDRconsistency}.
\etheos
We remark that unlike in the proof of Theorem \ref{Thm:HDRconsistency}, one can relax the assumption of $m$ and $\sigma$ being deterministic sequences depending on the smoothness level in the proof of Theorem \ref{Thm:HDR_rate}.  Also, both parts 1 and 2 of Theorem \ref{Thm:HDR_rate} are valid with a Gamma prior on $\sigma$ and a Poisson prior  on $m$.  In fact, part 2. of Theorem \ref{Thm:HDR_rate}  is  valid with even less restrictive conditions on the tail of the prior on $\sigma$. Refer to  \S 4.2 of \cite{bhattacharya2016bayesian} for an analogous example on density estimation.

\subsection{Proof of Corollary~\ref{Cor::quasirate}}
Choose $\rho(\cdot)=\Pi(\cdot\,|\, B_{n}(\thetastar,\varepsilon_{n}))$ in Theorem~\ref{Thm:quasirate}, we obtain that
\begin{equation}\label{Eqn:oracle_with_Bn}
\begin{aligned}
 \int \frac{1}{n} D^{(n)}_{\alpha}(\theta, \thetastar) \Pi_{\alpha}(d\theta\,|\, X^{(n)}) \le &\,\frac{\alpha}{n (1 - \alpha)} \int_{B_{n}(\thetastar,\varepsilon_{n})} r_n(\theta,\thetastar) \,\rho(d\theta) \\
 & + \Big\{ -\frac{1}{n (1 - \alpha)} \log \Pi(B_{n}(\thetastar,\varepsilon_{n})) \Big\} + \frac{1}{n(1-\alpha)}  \log(1/\varepsilon)
\end{aligned}
\end{equation}
holds with $\Prob^{(n)}_{\theta^\ast}$ probability at least $(1 - \varepsilon)$.

Now we will apply Chebyshev's inequality to obtain a high probability upper bound to the integral $\int_{B_{n}(\thetastar,\varepsilon_{n})} r_n(\theta,\thetastar) \,\rho(d\theta)$.
By applying Fubini's theorem and invoking the definition of $B_n(\theta^\ast,\varepsilon_n)$, we obtain the following bound for its expectation,
\begin{align*}
\bbE_{\theta^\ast}^{(n)}\Big[\int_{B_{n}(\thetastar,\varepsilon_{n})} r_n(\theta,\theta^\ast) \, \rho(d\theta)\Big] = \int_{B_{n}(\thetastar,\varepsilon_{n})}  \bbE_{\theta^\ast}^{(n)}\big[ r_n(\theta,\theta^\ast) \big]\, \rho(d\theta)\leq n\,\varepsilon_n^2.
\end{align*} 
Similarly, by applying the Cauchy-Schwarz inequality and Fubini's theorem, we have the following bound for its variance,
\begin{align*}
\mbox{Var}^{(n)}_{\theta^\ast}\Big[\int_{B_{n}(\thetastar,\varepsilon_{n})} r_n(\theta,\theta^\ast) \, \rho(d\theta)\Big] \leq \int_{B_{n}(\thetastar,\varepsilon_{n})}  \mbox{Var}_{\theta^\ast}^{(n)}\big[ r_n(\theta,\theta^\ast) \big]\, \rho(d\theta)\leq n\,\varepsilon_n^2.
\end{align*} 
Thus, by applying Cauchy-Schwarz inequality, we obtain that for any $D>1$,
\begin{align*}
& \bbP^{(n)}_{\theta^\ast}\bigg( \int_{B_{n}(\thetastar,\varepsilon_{n})} r_n(\theta,\theta^\ast) \, \rho(d\theta) > D\,n\, \varepsilon_n^2 \bigg) \\
& \leq  \bbP^{(n)}_{\theta^\ast}\bigg\{  \int_{B_{n}(\thetastar,\varepsilon_{n})} r_n(\theta,\theta^\ast)\,  \rho(d\theta) - \bbE_{\theta^\ast}^{(n)}\Big[  \int_{B_{n}(\thetastar,\varepsilon_{n})} r_n(\theta,\theta^\ast) \, \rho(d\theta)\Big]> (D-1)n \,\varepsilon_n^2\bigg\}  \\
& \leq \frac{\mbox{Var}^{(n)}_{\theta^\ast}  \big[ \int_{B_{n}(\thetastar,\varepsilon_{n})} r_n(\theta,\theta^\ast)\,  \rho(d\theta)\big]^2 }{(D-1)^2\, n^2 \,\varepsilon_n^4} 
\le \frac{1}{(D-1)^2\, n\, \varepsilon_n^2}.
\end{align*}

Putting all pieces together, we obtain by choosing $\varepsilon=\alpha \, n\,\varepsilon_n^2$ in~\eqref{Eqn:oracle_with_Bn} that with probability at least $1- \frac{1}{(D-1)^2\, n\, \varepsilon _{n}^2} - e^{-\alpha\, n\, \varepsilon _{n}^2} \leq 1 -\frac{2}{(D-1)^2\, n\, \varepsilon _{n}^2}$, it holds that
\begin{align*}
 \int \frac{1}{n}D^{(n)}(\theta, \thetastar)\, \Pi_{\alpha}(d\theta\,|\, X^{(n)}) \le \frac{(D + 1) \alpha}{1 - \alpha}  \varepsilon_n^2  + \Big\{ -\frac{1}{n (1 - \alpha)} \log \Pi(B_{n,}(\thetastar,\varepsilon_{n})) \Big\}.
\end{align*}

\subsection{Proof of Theorem~\ref{Thm:Main}}

\paragraph{Part 1 on regular posteriors:}
Fixed some constant $M>0$ that is sufficiently large. For $\lambda<\lambdastar$, let $\mathcal{F}_{n,\lambda}$ be the sieve sequence in Assumption B4 for $\Theta=\Theta_\lambda$ and $\Pi=\Pi_\lambda$ with $\varepsilon =\delta_n$; and for $\lambda > \lambdastar$, let $\mathcal{F}_{n,\lambda}$ be the sieve sequence in Assumption B4 for $\Theta=\Theta_\lambda$ and $\Pi=\Pi_\lambda$ with $\varepsilon = M\uepsilon_{n,\lambda}$. 

We will make use of the following two lemmas. Proofs of them can be found, for example, in \cite{ghosal2000} or \cite{ghosal2007}.

\blems
\label{Lemma:test}
Under Assumption B4, for each $\lambda \in \Lambda$, there exists constant $b>0$ and a test $\phi_{n,\lambda}$ such that: 
\begin{align*}
&\mbox{For $\lambda <\lambdastar$}:\quad \Prob_{\thetastar}^{(n)} \phi_{n,\lambda} \leq  e^{-b\, n\, \delta_n^2},\quad
\sup_{\theta \in\mathcal{F}_{n,\lambda}:\, d_n(\theta,\thetastar) \geq \delta_n} \Prob_{\theta}^{(n)}(1-\phi_{n,\lambda})\leq e^{-b\,n\,\delta_n^2}; \\
&\mbox{For $\lambda > \lambdastar$}:\quad \Prob_{\thetastar}^{(n)} \phi_{n,\lambda} \leq e^{-b\, M^2\, n\, \uepsilon_{n,\lambda}^2},\quad
\sup_{\theta \in\mathcal{F}_{n,\lambda}:\, d_n(\theta,\thetastar) \geq M\,\uepsilon_{n,\lambda}} \Prob_{\theta}^{(n)}(1-\phi_{n,\lambda})\leq e^{-b\,M^2\,n\,\uepsilon_{n,\lambda}^2}.
\end{align*}
\elems

\blems
\label{Lemma:denominator}
For every $\varepsilon>0$, we have, for any $C>0$,
\[
\Prob_{\thetastar}^{(n)}\bigg(\int_{B_n(\thetastar,\, \varepsilon))}\frac{p^{(n)}_{\theta}(X^{(n)})}{p^{(n)}_{\thetastar}(X^{(n)})}\, \Pi_{\lambdastar}(d\theta)\leq\, e^{-(1+C)\, n\, \varepsilon^2}\, \Pi_\lambdastar (B_n(\thetastar,\, \varepsilon)) \bigg)
\leq\frac{1}{C^2\, n\, \varepsilon^2}.
\]
\elems

Let test $\phi_n$ defined by $\phi_n = \max_\lambda \phi_{n,\lambda}$,
where $\phi_{n,\lambda}$ are the test functions given by Lemma~\ref{Lemma:test}.
Then we have the following decomposition:
\begin{align}
&1- \Prob_{\thetastar}^{(n)}\big[\Pi(\MODEL_{\lambdastar}\, | \, X_1,\ldots,X_n)\big]=
\Prob_{\thetastar}^{(n)}\big[\sum_{\lambda\neq \lambdastar} \Pi(\MODEL_\lambda\, | \, X_1,\ldots,X_n)\big]\notag\\
 = &\, \Prob_{\thetastar}^{(n)}\big[\sum_{\lambda\neq \lambdastar} \Pi(\MODEL_\lambda\, | \, X_1,\ldots,X_n)\, \phi_n\big] + \Prob_{\thetastar}^{(n)}\big[\sum_{\lambda\neq \lambdastar} \Pi(\MODEL_\lambda\, | \, X_1,\ldots,X_n)\, (1-\phi_n)\big]. \label{Eqn:goal}
\end{align}
We will bound the two terms on the r.h.s.~separately.

\

Applying Lemma~\ref{Lemma:test}, we can bound the first term as
\begin{equation}
\label{Eqn:term1}
\begin{aligned}
&\Prob_{\thetastar}^{(n)}\big[\sum_{\lambda\neq \lambdastar} \Pi(\MODEL_\lambda\, | \, X_1,\ldots,X_n)\, \phi_n\big] \\
&\leq \Prob_{\thetastar}^{(n)}\big[ \phi_n\big] \leq \sum_{\lambda\neq \lambda^\ast} \Prob_{\thetastar}^{(n)}\big[ \phi_{n,\lambda}\big]  \leq \sum_{\lambda <\lambdastar} e^{-b\, n\,\delta_n^2} +  \sum_{\lambda >\lambdastar} e^{-bM^2\, n\,\varepsilon_{n,\lambda}^2} \to 0,
\end{aligned}
\end{equation}
as $n\to\infty$, where we used Assumption B4 in the last step. 

Now we focus on the second term. Let $A_n$ denote the events in Lemma~\ref{Lemma:denominator} with $\varepsilon=\uepsilon_{n,\lambdastar}$. 
Then, we have the following decomposition 
\begin{align}
&\Prob_{\thetastar}^{(n)}\big[\sum_{\lambda\neq \lambdastar} \Pi(\MODEL_\lambda\, | \, X_1,\ldots,X_n)\, (1-\phi_n)\big] \notag\\
=&\, \Prob_{\thetastar}^{(n)}\big[\sum_{\lambda\neq \lambdastar} \Pi(\MODEL_\lambda\, | \, X_1,\ldots,X_n)\, (1-\phi_n)\, I(A_n^c) \big] + \Prob_{\thetastar}^{(n)}\big[\sum_{\lambda\neq \lambdastar} \Pi(\MODEL_\lambda\, | \, X_1,\ldots,X_n)\, (1-\phi_n)\, I(A_n) \big] \notag\\
=&\, \Prob_{\thetastar}^{(n)}\big[\sum_{\lambda\neq \lambdastar} \Pi(\MODEL_\lambda\, | \, X_1,\ldots,X_n)\, (1-\phi_n)\, I(A_n^c) \big]  + \Prob_{\thetastar}^{(n)}\big[\sum_{\lambda\neq \lambdastar} \Pi(\mathcal{F}_{n,\lambda}^c, \,\MODEL_\lambda\, | \, X_1,\ldots,X_n)\, (1-\phi_n)\, I(A_n) \big] \notag\\
&\, + \Prob_{\thetastar}^{(n)}\big[\sum_{\lambda> \lambdastar} \Pi( \{d_n(\theta,\,\thetastar) \geq M\, \uepsilon_{n,\lambda}\} \cap \mathcal{F}_{n,\lambda}, \MODEL_\lambda\, | \, X_1,\ldots,X_n)\, (1-\phi_n)\, I(A_n) \big] \notag\\
&\, +  \Prob_{\thetastar}^{(n)}\big[\sum_{\lambda> \lambdastar} \Pi( \{d_n(\theta,\,\thetastar) < M\, \uepsilon_{n,\lambda}\} \cap \mathcal{F}_{n,\lambda}, \MODEL_\lambda\, | \, X_1,\ldots,X_n)\, (1-\phi_n)\, I(A_n) \big] \notag\\
&\, + \Prob_{\thetastar}^{(n)}\big[\sum_{\lambda< \lambdastar} \Pi( \{d_n(\theta,\,\thetastar) \geq \delta_n\} \cap \mathcal{F}_{n,\lambda}, \MODEL_\lambda\, | \, X_1,\ldots,X_n)\, (1-\phi_n)\, I(A_n) \big] \notag\\
&\, +  \Prob_{\thetastar}^{(n)}\big[\sum_{\lambda< \lambdastar} \Pi( \{d_n(\theta,\,\thetastar) <\delta_n\} \cap \mathcal{F}_{n,\lambda}, \MODEL_\lambda\, | \, X_1,\ldots,X_n)\, (1-\phi_n)\, I(A_n) \big] \notag\\
:\,=&\, S_1 + S_2 +S_3+S_4+S_5+S_6.  \label{Eqn:term2}
\end{align}
Now we bound $S_1$-$S_6$ respectively:

\paragraph{Term $S_1$:} 
Since $\sum_{\lambda\neq \lambdastar} \Pi(\MODEL_\lambda\, | \, X_1,\ldots,X_n)\, (1-\phi_n) \leq 1$, we have
\begin{align*}
S_1 \leq \Prob_{\thetastar}^{(n)} (A_n^c) \leq \frac{1}{C^2\,n\,\uepsilon_{n,\lambdastar}^2} \to 0, \quad\mbox{as $n\to\infty$},
\end{align*}
by applying Lemma~\ref{Lemma:denominator}.

\paragraph{Term $S_2$:}
By Bayes' rule, for each $\lambda \neq\lambdastar$, we have
\begin{align}
\Pi(\mathcal{F}_{n,\lambda}^c,\, \MODEL_\lambda\, | \, X_1,\ldots,X_n)&= \frac{\pi_\lambda \int_{\mathcal{F}_{n,\lambda}^c}\frac{p^{(n)}_{\theta}X^{(n)})}{p^{(n)}_{\thetastar}(X^{(n)})}\, \Pi_\lambda(d\theta)}{\sum_{\lambda\in\Lambda} \pi_\lambda \int\frac{p^{(n)}_{\theta}X^{(n)})}{p^{(n)}_{\thetastar}(X^{(n)})}\, \Pi_\lambda(d\theta)} \notag\\
&\leq \frac{\pi_\lambda\,\int_{\mathcal{F}_{n,\lambda}^c}\frac{p^{(n)}_{\theta}X^{(n)})}{p^{(n)}_{\thetastar}(X^{(n)})}\, \Pi_\lambda(d\theta)}{\pi_{\lambdastar}\int\frac{p^{(n)}_{\theta}X^{(n)})}{p^{(n)}_{\thetastar}(X^{(n)})}\, \Pi_{\lambdastar}(d\theta)} \label{Eqn:S2a}
\end{align}
Applying Lemma~\ref{Lemma:denominator} and using Assumption B1, under the event $A_n$ the denominator on the r.h.s~can be bounded as
\begin{align}
\pi_{\lambdastar}\int\frac{p^{(n)}_{\theta}X^{(n)})}{p^{(n)}_{\thetastar}(X^{(n)})} (X_i)\, \Pi_{\lambdastar}(d\theta) \geq e^{-\, n\, \uepsilon_{n,\lambdastar}^2}\, e^{-(1+C)\, n\, \uepsilon_{n,\lambdastar}^2} e^{-c\,n\, \uepsilon_{n,\lambdastar}^2} = e^{-(2+C +c)\, n\, \uepsilon_{n,\lambdastar}^2}. \label{Eqn:denominator}
\end{align}
Combining inequalities~\eqref{Eqn:S2a} and  \eqref{Eqn:denominator}, we have
\begin{equation}
\label{Eqn:S2bound}
\begin{aligned}
S_2 & \leq e^{(2+C +c)\, n\, \uepsilon_{n,\lambdastar}^2} \,  \sum_{\lambda\neq \lambdastar} \pi_\lambda\,  \Prob_{\thetastar}^{(n)} \big[\int_{\mathcal{F}_{n,\lambda}^c}\frac{p^{(n)}_{\theta}X^{(n)})}{p^{(n)}_{\thetastar}(X^{(n)})}\, \Pi_\lambda(d\theta)\big]\\
&\overset{(i)}{=} e^{(2+C +c)\, n\, \uepsilon_{n,\lambdastar}^2} \,  \Big\{\sum_{\lambda< \lambdastar} \pi_\lambda\,  \Pi_\lambda(\mathcal{F}_{n,\lambda}^c)+ \sum_{\lambda> \lambdastar} \pi_\lambda\,  \Pi_\lambda(\mathcal{F}_{n,\lambda}^c)\Big\}\\
&\overset{(ii)}{\leq} e^{(2+C +c)\, n\, \uepsilon_{n,\lambdastar}^2} \,  \Big\{\sum_{\lambda< \lambdastar} \pi_\lambda\,  e^{-D\, n\,\delta_n^2} + \sum_{\lambda> \lambdastar} \pi_\lambda\,  e^{-D\,M^2\,n\,\uepsilon_{n,\lambda}^2}\Big\}\\
&\overset{(iii)}{\leq}  e^{(2+C +c -D\,M^2/2 )\, n\, \uepsilon_{n,\lambdastar}^2} \to 0, \quad\mbox{as $n\to\infty$}
\end{aligned}
\end{equation}
when $M$ is sufficiently large. Here in step (i) we applied Fubini's theorem, in step (ii) we used our definition of $\mathcal{F}_{n,\lambda}$ and Assumption B4, and step (iii) follows since by Assumptions B2 and B3, $\uepsilon_{n,\lambda} \geq \uepsilon_{n,\lambdastar}$ for $\lambda>\lambdastar$ and $\delta_n\geq M\varepsilon_{n,\lambdastar}$, and by Assumption B4, we have $\sum_{\lambda< \lambdastar} e^{-C\, n\,\delta_n^2} + \sum_{\lambda>\lambda} e^{-C\, n\,\varepsilon_{n,\lambda}^2} \leq 1$ with $C\geq D\, M^2/2$.

\paragraph{Term $S_3$:} 
Using Bayes' rule and inequality~\eqref{Eqn:denominator}, we can bound $S_3$ as (similar to the first two steps in the series of inequalities in~\eqref{Eqn:S2bound})
\begin{align*}
S_3 &\leq e^{(2+C +c)\, n\, \uepsilon_{n,\lambdastar}^2}   \sum_{\lambda> \lambdastar} \pi_\lambda\,  \Prob_{\thetastar}^{(n)} \big[(1-\phi_n) \,\int_{\{d_n(\theta,\,\thetastar) > M\, \uepsilon_{n,\lambda}\}\cap \mathcal{F}_{n,\lambda}}\frac{p^{(n)}_{\theta}X^{(n)})}{p^{(n)}_{\thetastar}(X^{(n)})}\, \Pi_\lambda(d\theta)\big]\\
&\leq e^{(2+C +c)\, n\, \uepsilon_{n,\lambdastar}^2}   \sum_{\lambda> \lambdastar} \pi_\lambda\,  \sup_{\theta\in \mathcal{F}_{n,\lambda}:\, d_n(\theta,\,\thetastar) > M\, \uepsilon_{n,\lambda}} \Prob_{\thetastar}^{(n)} \big[(1-\phi_n) \big]\\
&\overset{(i)}{\leq}  e^{(2+C +c)\, n\, \uepsilon_{n,\lambdastar}^2} \sup_{\lambda>\lambdastar} e^{-b\,M^2\,n\,\uepsilon_{n,\lambda}^2} \\
&\overset{(ii)}{\leq} Ce^{(2+C +c -bM^2)\, n\, \uepsilon_{n,\lambdastar}^2} \to 0, \quad\mbox{as $n\to\infty$},
\end{align*}
for sufficiently large $M$.
Here in step (i) we used Lemma~\ref{Lemma:test} and in step (ii) we used Assumption B2 that $\varepsilon_{n,\lambda}\geq \varepsilon_{n,\lambda}$ for $\lambda>\lambdastar$.

\paragraph{Term $S_4$:}
Similar as before, using Bayes' rule and inequality~\eqref{Eqn:denominator}, we can bound $S_4$ as
\begin{align*}
S_4& \leq e^{(2+C +c)\, n\, \uepsilon_{n,\lambdastar}^2}  \sum_{\lambda>\lambdastar} \pi_\lambda\,  \Pi_\lambda(d_n(\theta,\,\thetastar) < M\, \uepsilon_{n,\lambda})\\
& \overset{(i)}{\leq} e^{(2+C +c - H )\, n\, \uepsilon_{n,\lambdastar}^2} \to 0, \quad\mbox{as $n\to\infty$},
\end{align*}
for sufficiently large $H$.
Here in step (i) we used Assumption B3.

\paragraph{Term $S_5$:} 
Similar to the proof of the bound for term $S_3$, we have 
\begin{align*}
S_5 &\leq e^{(2+C +c)\, n\, \uepsilon_{n,\lambdastar}^2}   \sum_{\lambda< \lambdastar} \pi_\lambda\,  \sup_{\theta\in \mathcal{F}_{n,\lambda}:\, d_n(\theta,\,\thetastar) >  \delta_n} \Prob_{\thetastar}^{(n)} \big[(1-\phi_n) \big]\\
&\overset{(i)}{\leq} e^{(2+C +c)\, n\, \uepsilon_{n,\lambdastar}^2} e^{-b\,n\,\delta_n^2} \\
&\overset{(ii)}{\leq} e^{(2+C +c - bM^2)\, n\, \uepsilon_{n,\lambdastar}^2} \to 0, \quad\mbox{as $n\to\infty$},
\end{align*}
for sufficiently large $M$.
Here in step (i) we used Lemma~\ref{Lemma:test}, and in step (ii) we used Assumption B3 that $\delta_n\geq M \varepsilon_{n,\lambdastar}$.

\paragraph{Term $S_6$:} By Assumption B3, this term is zero.

\

\noindent Combining above bounds for terms $S_1$-$S_6$, inequality~\eqref{Eqn:term1}, and  equations~\eqref{Eqn:goal} and \eqref{Eqn:term2}, we can prove
$$
\Exs_{\thetastar}^{(n)}[ \Pi( \lambda=\lambdastar \, |\, X^{(n)}) ] \rightarrow 1,\quad\mbox{as $n\to\infty$}.$$

\

\paragraph{Part 2 on fractional posteriors:}
Similar to Lemma~\ref{Lemma:denominator} for regular posterior distributions, we have the following lemma for fractional posterior distributions.
\blems
\label{Lemma:denominator_b}
For every $\varepsilon>0$, we have, for any $C>0$,
\[
\Prob_{\thetastar}^{(n)}\bigg(\int_{B_n(\thetastar,\, \varepsilon))}\Big(\frac{p^{(n)}_{\theta}(X^{(n)})}{p^{(n)}_{\thetastar}(X^{(n)})}\Big)^{\alpha}\, \Pi_{\lambdastar}(d\theta)\leq\, e^{-\alpha\, (1+C)\, n\, \varepsilon^2}\, \Pi_\lambdastar (B_n(\thetastar,\, \varepsilon)) \bigg)
\leq\frac{1}{C^2\, n\, \varepsilon^2}.
\]
\elems

\noindent Let $A_n$ denote the complement of the event inside the bracket of the inequality in Lemma~\ref{Lemma:denominator_b} with $\varepsilon=\varepsilon_{n,\lambda^\ast}$. Then we have $\bbP_{\theta^\ast}^{(n)}(A_n) \geq 1-1/(C^2\,n\,\varepsilon_{n,\lambda^\ast}^2)$.

Note that we have the following decomposition 
\begin{align}
&\Prob_{\thetastar}^{(n)}\big[\sum_{\lambda\neq \lambdastar} \Pi_\alpha(\MODEL_\lambda\, | \, X_1,\ldots,X_n)\big] \notag\\
=&\, \Prob_{\thetastar}^{(n)}\big[\sum_{\lambda\neq \lambdastar} \Pi_\alpha(\MODEL_\lambda\, | \, X_1,\ldots,X_n)\, I(A_n^c)\big]\notag \\
 &\, +\Prob_{\thetastar}^{(n)}\big[\sum_{\lambda> \lambdastar} \Pi_\alpha( \{D^{(n)}_\alpha (\theta,\,\thetastar) \geq M\, n\,\uepsilon^2_{n,\lambda}\}, \, \MODEL_\lambda\, | \, X_1,\ldots,X_n)\, I(A_n) \big] \notag\\
&\, +  \Prob_{\thetastar}^{(n)}\big[\sum_{\lambda> \lambdastar} \Pi_\alpha( \{D^{(n)}_\alpha (\theta,\,\thetastar) < M\, n\,\uepsilon^2_{n,\lambda}\},\,  \MODEL_\lambda\, | \, X_1,\ldots,X_n)\, I(A_n) \big] \notag\\
&\, + \Prob_{\thetastar}^{(n)}\big[\sum_{\lambda< \lambdastar} \Pi_\alpha( \{D^{(n)}_\alpha (\theta,\,\thetastar) \geq n\,\delta^2_n\},\, \MODEL_\lambda\, | \, X_1,\ldots,X_n)\, I(A_n) \big] \notag\\
&\, +  \Prob_{\thetastar}^{(n)}\big[\sum_{\lambda< \lambdastar} \Pi_\alpha( \{D^{(n)}_\alpha (\theta,\,\thetastar) <n\,\delta^2_n\},\, \MODEL_\lambda\, | \, X_1,\ldots,X_n)\, I(A_n) \big] \notag\\
:\,=&\, S_1 + S_2 +S_3+S_4+S_5.  \label{Eqn:decomposition_frac}
\end{align}
Now we bound $S_1$-$S_5$ respectively:

\paragraph{Term $S_1$:} 
Since $\sum_{\lambda\neq \lambdastar} \Pi_\alpha(\MODEL_\lambda\, | \, X_1,\ldots,X_n) \leq 1$, we have
\begin{align*}
S_1 \leq \Prob_{\thetastar}^{(n)} (A_n^c) \leq \frac{1}{C^2\,n\,\uepsilon_{n,\lambdastar}^2} \to 0, \quad\mbox{as $n\to\infty$},
\end{align*}
by applying Lemma~\ref{Lemma:denominator_b}.

\paragraph{Term $S_2$:}
By Bayes' rule, for each $\lambda>\lambdastar$, we have
\begin{align}
\Pi_\alpha( \{D^{(n)}_\alpha (\theta,\,\thetastar) \geq M\, \uepsilon_{n,\lambda}\},\, \MODEL_\lambda\, | \, X_1,\ldots,X_n)&= \frac{\pi_\lambda \int_{ D^{(n)}_\alpha (\theta,\,\thetastar) \geq M\, n\,\uepsilon_{n,\lambda}^2}\Big(\frac{p^{(n)}_{\theta}X^{(n)})}{p^{(n)}_{\thetastar}(X^{(n)})}\Big)^\alpha\, \Pi_\lambda(d\theta)}{\sum_{\lambda\in\Lambda} \pi_\lambda \int\Big(\frac{p^{(n)}_{\theta}X^{(n)})}{p^{(n)}_{\thetastar}(X^{(n)})}\Big)^\alpha\, \Pi_\lambda(d\theta)} \notag\\
&\leq \frac{\pi_\lambda\,\int_{ D^{(n)}_\alpha (\theta,\,\thetastar) \geq M\,n\, \uepsilon^2_{n,\lambda}}\Big(\frac{p^{(n)}_{\theta}X^{(n)})}{p^{(n)}_{\thetastar}(X^{(n)})}\Big)^\alpha\, \Pi_\lambda(d\theta)}{\pi_{\lambdastar}\int\Big(\frac{p^{(n)}_{\theta}X^{(n)})}{p^{(n)}_{\thetastar}(X^{(n)})}\Big)^\alpha\, \Pi_{\lambdastar}(d\theta)} \label{Eqn:S2a_quasi}
\end{align}
Using the definition of the $\alpha$-divergence $D^{(n)}_\alpha$, we have that for any $\theta$,
\begin{align}\label{Eqn:quasi_key}
\bbE_{\theta^\ast}^{(n)} \Big[\Big(\frac{p^{(n)}_{\theta}(X^{(n)})}{p^{(n)}_{\thetastar}(X^{(n)})}\Big)^{\alpha}\Big] = A^{(n)}_{\alpha}(\theta,\theta^\ast) =\exp\big\{- (1-\alpha)\, D^{(n)}_\alpha(\theta,\theta^\ast)\big\}.
\end{align}

Applying Lemma~\ref{Lemma:denominator_b} and using Assumptions A2 and B3, under the event $A_n$ the denominator on the r.h.s~can be bounded as
\begin{align}
\pi_{\lambdastar}\int\Big(\frac{p^{(n)}_{\theta}X^{(n)})}{p^{(n)}_{\thetastar}(X^{(n)})}\Big)^\alpha\, \Pi_{\lambdastar}(d\theta) \geq e^{-\, n\, \uepsilon_{n,\lambdastar}^2}\, e^{-\alpha\, (1+C)\, n\, \uepsilon_{n,\lambdastar}^2} e^{-c\,n\, \uepsilon_{n,\lambdastar}^2} = e^{-(1+\alpha+\alpha\, C +c)\, n\, \uepsilon_{n,\lambdastar}^2}. \label{Eqn:denominator_quasi}
\end{align}
Combining inequalities~\eqref{Eqn:S2a_quasi}, \eqref{Eqn:quasi_key} and  \eqref{Eqn:denominator_quasi}, we have
\begin{equation}
\label{Eqn:S2bound_frac}
\begin{aligned}
S_2 & \leq e^{(1+\alpha+\alpha\, C +c)\, n\, \uepsilon_{n,\lambdastar}^2} \,  \sum_{\lambda > \lambdastar} \pi_\lambda\,  \Prob_{\thetastar}^{(n)} \big[\int_{D^{(n)}_\alpha (\theta,\,\thetastar) \geq M\, n\,\uepsilon^2_{n,\lambda}}\Big(\frac{p^{(n)}_{\theta}X^{(n)})}{p^{(n)}_{\thetastar}(X^{(n)})}\Big)^\alpha\, \Pi_\lambda(d\theta)\big]\\
&\leq e^{(1+\alpha+\alpha\, C +c)\, n\, \uepsilon_{n,\lambdastar}^2} \,  \sum_{\lambda>\lambdastar} \pi_\lambda\,e^{-n(1-\alpha) M\,\varepsilon_{n,\lambda}^2}\\
&\leq e^{(1+\alpha+\alpha\, C +c -(1-\alpha)\,M )\, n\, \uepsilon_{n,\lambdastar}^2} \to 0, \quad\mbox{as $n\to\infty$}
\end{aligned}
\end{equation}
when $M$ is sufficiently large.

\paragraph{Term $S_3$:}
Similar to the previous step of bounding $S_2$, using Bayes' rule and inequality~\eqref{Eqn:denominator_quasi}, we can bound $S_3$ as
\begin{align*}
S_3& \leq e^{(1+\alpha+\alpha\, C +c)\, n\, \uepsilon_{n,\lambdastar}^2}  \sum_{\lambda>\lambdastar} \pi_\lambda\,  \Pi_\lambda(D^{(n)}_\alpha(\theta,\,\thetastar) < M\,n\, \uepsilon_{n,\lambda}^2)\\
& \overset{(i)}{\leq}C\, e^{(1+\alpha+\alpha\, C +c - H )\, n\, \uepsilon_{n,\lambdastar}^2} \to 0, \quad\mbox{as $n\to\infty$},
\end{align*}
for sufficiently large $H$.
Here in step (i) we used Assumptions B2 and B3.

\paragraph{Term $S_4$:} 
Similar to the proof of the bound for term $S_2$, we have 
\begin{align*}
S_4 &\leq e^{(1+\alpha+\alpha\, C +c )\, n\, \uepsilon_{n,\lambdastar}^2}   \sum_{\lambda< \lambdastar} \pi_\lambda\,  e^{-n(1-\alpha) \,\delta_n^2} \\
&\leq e^{(1+\alpha+\alpha\, C +c  - (1-\alpha) \, M)\, n\, \uepsilon_{n,\lambdastar}^2} \to 0, \quad\mbox{as $n\to\infty$},
\end{align*}
for sufficiently large $M$.
Here in the last step we used Assumption B3.

\paragraph{Term $S_5$:} By Assumption B2,  this term is zero.

\

\noindent Combining above bounds for terms $S_1$-$S_5$ and \eqref{Eqn:decomposition_frac} yields a proof for
$$
\Exs_{\thetastar}^{(n)}[ \Pi_\alpha( \lambda=\lambdastar \, |\, X^{(n)}) ] \rightarrow 1,\quad\mbox{as $n\to\infty$}.$$

\subsection{Proof of Theorem~\ref{Thm:normalpostMSrate}}
For simplicity, we let
\begin{align*}
\varepsilon_{n}^2 :\,= \min_{\lambda\in\Lambda} \, \Big\{ \min_{\varepsilon_\lambda>0}\Big\{\varepsilon^2_{\lambda} + \Big(-\frac{1}{n}\log \Pi_\lambda(B_n(\theta^\ast,\,\varepsilon_\lambda))\Big)\Big\} + \Big(-\frac{1}{n}\log \pi_\lambda\Big)\Big\} + \frac{\log |\Lambda|}{n},
\end{align*}
and let $\lambda^\ast$ be the index under which the outside minimum in the display achieves, and $\varepsilon_{n,\lambda^\ast}$ be the $\varepsilon_\lambdastar$ that achieves the inner minimum under $\lambda=\lambda^\ast$.
According to Assumption A3$'$, for any $\varepsilon>0$ and any $\lambda$, there exists a sequence of sieves $\mathcal{F}_{n,\lambda}$ such that
\begin{align*}
\log N(\varepsilon, \, \mathcal{F}_{n,\lambda}, \, d_n) \leq n\,\varepsilon^2
\qquad\mbox{and}\qquad \Pi_\lambda(\mathcal{F}_{n,\lambda}^c) \leq e^{-D\,n\,\varepsilon^2}.
\end{align*}
We can use this to construct a sieve sequence $\{\mathcal{F}_{n}\}$ of the entire parameter space $\Theta=\bigcup_{\lambda\in\Lambda} \Theta_\lambda$ as
\begin{align*}
\mathcal{F}_{n} :\,=\bigcup_{\lambda}  \mathcal{F}_{n,\lambda},
\end{align*}
which satisfies
\begin{align*}
\log N(\varepsilon, \, \mathcal{F}_{n}, \, d_n) & \leq \log\Big(\sum_{\lambda} N(\varepsilon, \, \mathcal{F}_{n,\lambda}, \, d_n) \Big) \leq  n\,\varepsilon^2 + \log |\Lambda|, \mbox{\ \ and}\\
\Pi_\lambda(\mathcal{F}_{n}^c) & \leq  \inf_{\lambda\in\Lambda} \Pi_\lambda(\mathcal{F}_{n,\lambda}^c) \leq e^{-D\,n\,\varepsilon^2}.
\end{align*}
Therefore, there exists some sufficiently large constant $c$ such that for $\varepsilon= c\, \varepsilon_{n}$, we have
\begin{equation}\label{Eqn:NewA3}
\begin{aligned}
\log N(\varepsilon, \, \mathcal{F}_{n}, \, d_n) & \leq \log\Big(\sum_{\lambda} N(\varepsilon, \, \mathcal{F}_{n,\lambda}, \, d_n) \Big) \leq  2\,n\,\varepsilon^2, \mbox{\ \ and}\\
\Pi_\lambda(\mathcal{F}_{n}^c) & \leq \sum_{\lambda}  \Pi_\lambda(\mathcal{F}_{n,\lambda}^c) \leq e^{-\frac{1}{2}\, D\,n\,\varepsilon^2}.
\end{aligned}
\end{equation}
This can be considered as Assumption A3 with $\Theta$ as the parameter space and $\Pi=\sum_{\lambda} \pi_\lambda\, \Pi_\lambda$ as the prior over $\Theta$. Similarly, if we let
\begin{align*}
B_{n}(\thetastar,\varepsilon)=\{\theta\in\Theta:\, D(p_{\thetastar}^{(n)}, p_{\theta}^{(n)})\leq n\,\varepsilon^2,\, V(p_{\theta_0}^{(n)}, p_{\theta}^{(n)})\leq n\,\varepsilon^2\}, \quad \varepsilon>0.
\end{align*}
Then by the definition of the $\varepsilon_n$, we have
\begin{align}\label{Eqn:NewA2}
\Pi (B_{n}(\theta^\ast,\varepsilon_{n})) \geq \pi_{\lambdastar}\,\Pi_\lambdastar (B_n(\theta^\ast,\uepsilon_{n})) \geq \pi_{\lambdastar}\,\Pi_\lambdastar (B_n(\theta^\ast,\uepsilon_{n,\lambda^\ast}))  \geq e^{-n \, \varepsilon_{n}^2}.
\end{align}
This can be considered as Assumption A2 with $\Theta$ and $\Pi$. Now under the condition of the theorem, Assumption A1 applies for all $\theta_1\in\Theta_\lambda$ and any $\lambda \in\Lambda$, so it also applies for all $\theta_1\in \Theta$, which is Assumption A1 with $\Theta$ and $\Pi$. Now, we apply Theorem~\ref{Thm:rate} for $\Theta$ as the parameter space and $\Pi$ as the posterior distribution, we obtain that for sufficiently large $M$, it holds 
\begin{align*}
&\Exs_{\thetastar}^{(n)}[ \Pi( d_n(\theta, \,\thetastar) > M \, \varepsilon_{n}\, |\, X^{(n)}) ] \rightarrow 0,\quad \mbox{as $n\to\infty$.}
\end{align*}

\subsection{Proof of Theorem~\ref{Thm:quasipostMSrate}}
Similar to the proof of Theorem~\ref{Thm:normalpostMSrate}, we let
\begin{align*}
\varepsilon_{n}^2 :\,= \min_{\lambda\in\Lambda} \, \Big\{ \min_{\varepsilon_\lambda>0}\Big\{(\varepsilon^2_{\lambda} + \Big(-\frac{1}{n}\log \Pi_\lambda(B_n(\theta^\ast,\,\varepsilon_\lambda))\Big)\Big\} + \Big(-\frac{1}{n}\log \pi_\lambda\Big)\Big\},
\end{align*}
and define $\lambda^\ast$ and $\varepsilon_{n,\lambda^\ast}$ accordingly.
Let us write
\begin{align*}
U_n:\,= \Big\{\theta\in\Theta:\, D^{(n)}_\alpha(\theta,\,\theta^\ast) \geq \frac{M}{1-\alpha}\, n\,\varepsilon_n^2\Big\}.
\end{align*}
Then, we can express the desired posterior probability as
\begin{align}
\Pi_{\alpha} \Big(D^{(n)}_{\alpha}(\theta,\,\theta^\ast) \geq \frac{M}{1-\alpha}\, n\,\varepsilon_n^2\ \Big| \, X^{(n)}\Big)  =  \frac{\int_{U_n} e^{-\alpha\, r_{n}(\theta,\,\theta^\ast)} \,\Pi(d\theta)}{\int_{\Theta} e^{-\alpha\, r_{n}(\theta,\,\theta^\ast)}\, \Pi(d\theta)},\label{Eq:post}
\end{align} 
where recall that $r_{n}(\theta,\,\theta^\ast)$ is the negative log-likelihood ratio between $\theta$ and $\thetastar$.

Let us first consider the numerator.
By the definition of the $\alpha$-divergence $D^{(n)}_\alpha(\theta,\,\theta^\ast)$, we have 
\begin{align}\label{eq:key_div}
\bbE^{(n)}_{\thetastar} e^{-\alpha\, r_{n}(\theta,\,\theta^\ast)}  = A^{(n)}_{\alpha}(\theta, \theta^\ast) = e^{-(1 - \alpha)\,  D^{(n)}_{\alpha}(\theta,\,\theta^\ast)}.
\end{align}
Now integrating both side with respect to the prior $\Pi$ over $U_n$ and applying Fubini's theorem, we can get
\begin{align*}
\bbE^{(n)}_{\thetastar} \int_{U_n} e^{-\alpha\, r_{n}(\theta,\,\theta^\ast)} \Pi(d\theta)  = \int_{U_n}  e^{-(1 - \alpha)\, D^{(n)}_{\alpha}(\theta,\,\theta^\ast)} \Pi(d\theta) \leq e^{-M\, n \, \varepsilon_n^2},
\end{align*}
where the last step follows from the definition of $U_n$.
An application of the Markov inequality yields the following high probability bound for the numerator on the right hand side of~\eqref{Eq:post},
\begin{align} \label{Eqn:NumBound}
\bbP^{(n)}_{\thetastar}\Big[\int_{U_n} e^{-r_{n, \alpha}(f)} \Pi(df)  \geq e^{-M\, n \, \varepsilon_n^2/2} \Big] \leq e^{-M\, n \, \varepsilon_n^2/2} \leq \frac{4}{M^2 \,n\,\varepsilon_n^2}.
\end{align}

Next, we consider the denominator on the right hand side of~\eqref{Eq:post}. We always have the lower bound
\begin{align*}
\int_{\Theta} e^{-\alpha\, r_{n}(\theta,\,\theta^\ast)} \Pi(d\theta) \geq \int_{B_{n}(\theta^\ast,\varepsilon_n;)} e^{-\alpha\, r_{n}(\theta,\,\theta^\ast)} \Pi(d\theta).
\end{align*}
Since the prior concentration bound~\eqref{Eqn:NewA2} holds, we invoke Lemma~\ref{Lemma:denominator_b} to obtain
for any $D>1$, we have
\begin{align} \label{Eqn:DenBound}
\bbP^{(n)}_{\thetastar}\Big[\int_{B_{n}(\theta^\ast,\varepsilon_n)} e^{-\alpha\, r_{n}(\theta,\,\theta^\ast)} \Pi(d\theta) \leq e^{-\alpha \, D \,n\,\varepsilon_n^2}\Big] \leq \frac{1}{(D-1)^2n\,\varepsilon_n^2}.
\end{align}

Now combining \eqref{Eq:post}, \eqref{Eqn:NumBound} and ~\eqref{Eqn:DenBound}, we obtain that with probability at least $1 - 2/\{(M/4-1)^2 n \varepsilon_n^2\} \to 1$,
\begin{align*}
\Pi_{\alpha} \Big(D^{(n)}_{\alpha}(\theta,\,\theta^\ast) \geq \frac{M}{1-\alpha}\, n\,\varepsilon_n^2\ \Big| \, X^{(n)}\Big)
\leq  e^{-M\, n \, \varepsilon_n^2/2}  e^{M\, n \, \varepsilon_n^2/4} =  e^{-M\, n \, \varepsilon_n^2/4}\to 0,\quad\mbox{as $n\to\infty$}.
\end{align*}

\subsection{Proof of Theorem~\ref{Thm:quasipostMS_Oracle}}
Define the same $\varepsilon_{n}$, $\lambda^\ast$ and $\varepsilon_{n,\lambda^\ast}$ as in the proof of Theorem~\ref{Thm:quasipostMSrate}.
Then we have
\begin{align}\label{Eqn:quasiPriorConcen}
\Pi \big(\{\theta\in B_n(\thetastar,\varepsilon_{n,\lambdastar})\}\cap\{\lambda=\lambda^\ast\}\big) \geq \pi_{\lambdastar}\,\Pi_\lambdastar (B_n(\theta^\ast,\uepsilon_{n,\lambda^\ast})) \geq e^{-n\,\varepsilon_{n}^2}.
\end{align}
Now apply Theorem~\ref{Thm:PAC-Bayes-one-model} with $\Theta$ as the parameter space and 
$$\rho(\cdot)=\Pi\big(\cdot\,\big|\, \{\theta\in B_n(\thetastar,\varepsilon_{n,\lambdastar})\}\cap\{\lambda=\lambda^\ast\}\big),$$
we obtain that with $\Prob^{(n)}_{\theta^\ast}$ probability at least $1-\varepsilon$, 
\begin{align*}
 \int \frac{1}{n}D^{(n)}_\alpha(\theta, \thetastar) \Pi_\alpha(d\theta\,|\, X^{(n)}) \le&\,\frac{1}{n (1 - \alpha)} \Big\{\int_{B_n(\thetastar,\varepsilon_{n,\lambdastar})} r_n(\theta,\thetastar) \rho(d\theta) + \varepsilon_n^2\Big\} + \frac{1}{n(1-\alpha)} \log(1/\varepsilon),
\end{align*}
where we used inequality~\eqref{Eqn:quasiPriorConcen}.
The rest of the proof follows the same lines as the proof of Corollary~\ref{Cor::quasirate} to obtain a high probability upper bound to the integral $\int_{B_n(\thetastar,\varepsilon_{n,\lambdastar})} r_n(\theta,\thetastar)$, leading to that for some sufficiently large $M$, with $\Prob^{(n)}_{\theta^\ast}$ probability tending to one,
\begin{align*}
 \int \frac{1}{n} D^{(n)}_\alpha(\theta, \thetastar) \Pi_\alpha(d\theta\,|\, X^{(n)}) \le&\,\frac{M}{1 - \alpha} \Big\{ \varepsilon_{n,\lambdastar}^2 
 + \Big( -\frac{1}{n}  \log \pi_{\lambdastar} -\frac{1}{n} \log \Pi(B_n(\thetastar,\varepsilon_{n,\lambdastar})) \Big),
\end{align*}
which implies the desired Bayesian oracle inequality.

\subsection{Proof of Proposition~\ref{Prop:NPB2}}

We use the fact that the projection of any function $f\in L_2(\mu^n; \,[0,1]^p)$ onto $\Theta_I$ is the conditional expectation $\bbE[f\,|\, X_I]$. Therefore, for any $I \not\supset I^\ast$, we have
\begin{align*}
d_n^2(f^\ast, \Theta_I) = \bbE[f^\ast - \bbE[f^\ast\, |\, X_I] ]^2 = \bbE[ \mbox{Var}[f^\ast(X_{I^\ast})\, | \, X_{I}]].
\end{align*}
Write $I = I_1 \cup I_2$ where $I_1 \subset I^\ast$ and $I_2 \cap I^\ast =\emptyset$. Then $I \not\supset I^\ast$ implies $I_1 \subsetneq I$.
By adding and subtracting $\bbE[f^\ast\, |\, X_{I_1}]$ in the first expression of $d_n^2(f^\ast, \Theta_I)$, we have
\begin{align}
 &d_n^2(f^\ast, \Theta_I) = \bbE[f^\ast - \bbE[f^\ast\, |\, X_I] ]^2 = \bbE[f^\ast -\bbE[f^\ast\, |\, X_{I_1}] + \bbE[f^\ast\, |\, X_{I_1}] - \bbE[f^\ast\, |\, X_I] ]^2\notag\\
 =&\,\bbE[f^\ast -\bbE[f^\ast\, |\, X_{I_1}]]^2 + 2\bbE[(f^\ast -\bbE[f^\ast\, |\, X_{I_1}])(\bbE[f^\ast\, |\, X_{I_1}] - \bbE[f^\ast\, |\, X_I] )] + \bbE[\bbE[f^\ast\, |\, X_{I_1}] - \bbE[f^\ast\, |\, X_I] ]^2\notag\\
 \overset{(i)}{=}&\,\bbE[f^\ast -\bbE[f^\ast\, |\, X_{I_1}]]^2+ \bbE[\bbE[f^\ast\, |\, X_{I_1}] - \bbE[f^\ast\, |\, X_I] ]^2\notag \\
 \geq &\, \bbE[f^\ast -\bbE[f^\ast\, |\, X_{I_1}]]^2 = d_n^2(f^\ast,\Theta_{I_1}). \label{Eqn:dn_compare}
\end{align}
Here, in step (i), we used the fact that under Assumption NP-C, the cross term
\begin{align*}
&2\bbE[(f^\ast -\bbE[f^\ast\, |\, X_{I_1}])(\bbE[f^\ast\, |\, X_{I_1}] - \bbE[f^\ast\, |\, X_I] )] \\
\overset{(ii)}{=}&\, 2\bbE[(f^\ast -\bbE[f^\ast\, |\, X_{I_1}]) \, \bbE[(\bbE[f^\ast\, |\, X_{I_1}] - \bbE[f^\ast\, |\, X_I] )\, |\, X_{I^\ast}]]\\
=&\, 2\bbE[(f^\ast -\bbE[f^\ast\, |\, X_{I_1}]) \, (\bbE[f^\ast\, |\, X_{I_1}] - \bbE[\bbE[f^\ast\, |\, X_I] \, |\, X_{I^\ast}])] \\
\overset{(iii)}{=}&\,  2\bbE[(f^\ast -\bbE[f^\ast\, |\, X_{I_1}]) \, (\bbE[f^\ast\, |\, X_{I_1}] - \bbE[f^\ast\, |\, X_{I_1}] )] =0,
\end{align*}
where in step (ii) we used the law of iterated expectations and the fact that $(f^\ast -\bbE[f^\ast\, |\, X_{I_1}])$ is a deterministic function of $X_{I^\ast}$ so that this term can be pulled out from the inner conditional expectation, and step (iii) follows by using the fact that if $X$, $Y$, $Z$ and $W$ are independent, then for any measurable function $f(x,y)$, $\bbE[\bbE[f(X,Y)\,|\, X, Z]\,|\, X, W] = \bbE[f(X, Y)\,|\, X]$.

Inequality~\eqref{Eqn:dn_compare} implies that 
\begin{align*}
\inf_{I \not\supset I^\ast} d_n^2(f^\ast, \Theta_I)  = \inf_{I \subsetneq I^\ast} d_n^2(f^\ast, \Theta_I).
\end{align*}
Combining this with the monotonicity of $d_n^2(f^\ast, \Theta_I)$ in $I$, we proved that
\begin{align*}
\inf_{I \not\supset I^\ast} d_n^2(f^\ast, \Theta_I) = \min_{j=1,\ldots,d^\ast}\bbE[f^\ast - \bbE[f^\ast\, |\, X_{I^\ast\setminus\{j\}}] ]^2 
=\min_{j=1,\ldots,d^\ast}
\bbE[ \mbox{Var}[f^\ast(X_{I^\ast})\, | \, X_{I^\ast\setminus\{j\}}]].
\end{align*}
The second part follows from the identity
\begin{align*}
\bbE[f^\ast - \bbE[f^\ast\, |\, X_{I^\ast\setminus\{j\}}] ]^2 =  \sum_{v\in \mathbb{N}_0^{I^\ast}, v_{j} \neq 0} \langle f^*, e_v\rangle^2,
\end{align*}
which can be verified by direct calculation and noticing the fact that $e_v$ forms an orthonormal bases and $e_0\equiv 1$ so that $\bbE[e_{v_j}(X_j)] = 0$ for any $j\in [p]$.

\subsection{Proof of Theorem~\ref{Thm:HNRconsistency}}
For simplicity, we only prove the part for the regular posterior distribution, and a proof for the fractional posterior distribution is almost the same. We prove this by verifying the assumptions in Theorem~\ref{Thm:Main} under the regression setup.

First we verify Assumption B1 and B4. 
By the choice of the prior over the model space, we have
\begin{align*}
\pi_{I^\ast} \geq c \, p^{-d^\ast} (1-p^{-1})^{p-d^\ast} \geq c' \, p^{-d^\ast} \geq c' \exp\big\{-d_0\log p\} \geq e^{-n\,\varepsilon_{n,I^\ast}^2},
\end{align*} 
implying the model space prior concentration in Assumption B1.
Using the result in Section 5.1 in \cite{vandervaart2009}, Assumption A2 is true with 
\begin{align*}
\varepsilon_{n,I} = n^{-\frac{\beta}{2\beta + |I|}} \, (\log n)^{\frac{4\beta +|I|}{4\beta +2 |I|}} \wedge \sqrt{\frac{d_0\log p}{n}},
\end{align*}
implying the parameter space prior concentration in Assumption B1. For regression model with random design, the likelihood ratio test for $\thetastar = f^\ast$ versus $\theta_1=f_1$ satisfies Assumption A1 with $d_n(f_1,f_2)=\|f_1-f_2\|_{\mu,2}$ when the function class is uniformly bounded by some constant, which is satisfied with our modified conditional prior $\Pi^{\mathrm{GP}}$. 
Equations (5.6) and (5.7) in \cite{yang2015b} implies the existence of the sieve sequence in Assumption A3 for any $\varepsilon>0$ and any model index $I \subset p$ satisfying $|I| \leq d_0$. 
These two arguments verified B4.

Assumption B3 is true by the condition on $\delta_n$ is the statement of the theorem and our choice of $\varepsilon_{n,I}$. The only remaining part is to verify Assumption B2. We invoke the following lemma, whose proof is provided in the next subsection. We remark that the proof of this lemma can be generalized to any RKHS associated with a stationary GP, which in turn can be converted into a anti-concentration inequality for the small ball probability of a GP. As an intermediate result, Lemma~\ref{Eqn:EigenSystem} in the proof of the following lemma, playing a key role that characterizes the eigensystem of any one-dimensional stationary kernel over $[0,1]^2$, is interest in its own right. In Appendix~\ref{App:eigen}, we characterize eigensystems of some popular stationary covariance kernels. 

\blems
\label{Lemma:Entropylower}
Under Assumption NP-C, if $a \geq 2$, then for any $\varepsilon \in (0,\, a^{-|I|/2})$, we have
\begin{align*}
\log N(\varepsilon,\bbH^a_I, d_n)\geq C\frac{a^{|I|}}{|I|^{|I|}} \log\Big(\frac{1}{\varepsilon\,  a^{|I|/2}}\Big)^{\frac{|I|+2}{2}}.
\end{align*}
\elems

Let $a_{n,I}:=n^{\frac{1}{2\beta +|I|}}$ denote the lower bound on the constraint of the conditional prior of $A$.
Under the condition that $\beta \geq d_0/2$, we always have 
\begin{align*}
\varepsilon_{n,I}  \leq a_{n, I}^{-|I|/2}
\end{align*}
 and the constraint $A\geq n^{\frac{1}{2\beta+ |I|}}$ in our conditional prior
for all $I$ and $a$ in the support of the prior of $A$ given $I$. According to Lemma 4.7 in \cite{vandervaart2009},  when $a\geq a_{n,I}$, we always have $\bbH^a_I \supset \bbH^{a_{n,I}}_I$ and therefore $\log N(\varepsilon,\bbH^a_I, d_n) \geq \log N(\varepsilon,\bbH^{a_{n,I}}_I, d_n)$ for any $\varepsilon>0$.
Using this fact, and applying Lemma~\ref{Lemma:Entropylower} and the result on the relation between the small ball probability of GP and the covering entropy of its associated RKHS \cite{vandervaart2008}, we obtain that for all $I$ such that $\varepsilon_{n,I}  \geq \varepsilon_{n,I^\ast}$,
\begin{align*}
\Pi_I(d_n(f,f^\ast) \leq M \varepsilon_{n,I} ) \leq \sup_{a\geq n^{1/(2\beta+ |I|)}} \Pi_I(d_n(f,f^\ast) \leq M \varepsilon_{n,I} \, |\, A=a) \leq e^{-C n^{\frac{|I|}{2\beta + |I|}}} = e^{-C\,n\,\varepsilon_{n,I}^2} \leq e^{-H \,n\,\varepsilon_{n,I^\ast}^2}
\end{align*}
holds for any $H>0$ when $n$ is sufficiently large, which verifies Assumption B2.

\subsection{Proof of Lemma~\ref{Lemma:Entropylower}}
By Assumption NP-C, $d_n(f_1,f_2)=\|f_1-f_2\|_{\mu,2}$ can be bounded below up to some multiplicative constant by $\|f_1-f_2\|_{U,2}$, where $U$ is the uniform distribution over $[0,1]$. Therefore, it suffices to prove a lower bound on the covering entropy with respect to the $\|\cdot\|_{U,2}$ metric.

According to the argument before Lemma~\ref{Eqn:EigenSystem},  the orthonormal eigenbasis of $K_I^a$ as a kernel over $[0,1]^I$ is
the tensor product of the one-dimensional one as $\{\phi_v(x) = \prod_{j\in I}\phi_{v_j}(x_j): v \in \mathbb{N}_0^{I}\}$, with the corresponding eigenvalue of
$\phi_v$ being $\eta_v:\,= \prod_{j\in I} \eta_{v_j}$. Under these notation, the unit ball $\bbH^a_I$ in the RKHS can be identified with the ellipsoid in $\ell_2$ as
\begin{align*}
\mathcal{E} = \{u=(u_v,\, v\in \mathbb{N}_0^{I}):\, \sum_{v} \frac{u_v^2}{\eta_v} \leq 1\}.
\end{align*}
The desired bound becomes
\begin{align*}
\log N(\varepsilon, \,\mathcal{E}, \,\|\cdot\|_{\ell_2})\geq C a^{|I|}\log\Big(\frac{1}{\varepsilon\,  a^{|I|/2}}\Big)^{\frac{|I|+2}{2}}.
\end{align*}
For any subset $V$ of $\mathbb{N}_0^{I}$, the following is true using the volume argument 
\begin{align}
\label{Eqn:VolArg}
N(\varepsilon, \,\mathcal{E}, \,\|\cdot\|_{\ell_2}) \geq \frac{\mbox{Vol}(\mathcal{E}_V)}{\mbox{Vol}(B_V)} \,\Big(\frac{1}{\varepsilon}\Big)^{|V|},
\end{align}
where $\mathcal{E}_V$ stands for coordinate projection of $\mathcal{E}$ onto $V$ and $B_V$ stands for the unit ball in $\bbR^{|V|}$.

\noindent By choosing $V=\{v\in \mathbb{N}_0^{I}:\, v_j \leq m -1, \, \forall j\in I\}$ for some $m\leq a^2$, we have
\begin{align*}
\frac{\mbox{Vol}(\mathcal{E}_V)}{\mbox{Vol}(B_V)} = \prod_{v\in V} \eta_v^{1/2} = \Big(\prod_{j=0}^{m-1} \eta_j^{1/2}\Big)^{|I|\,m^{|I|-1}} \asymp a^{-\frac{1}{2} |I|\, m^{|I|}} e^{-\frac{|I|\,m^{|I|+2}}{8a^2}},
\end{align*}
where we have used Proposition~\ref{Prop:squareexp} for expressions of eigenvalues $\{\eta_j:\,j\leq m-1\}$, and the exponent $|I|\,m^{|I|-1}$ is because every $\eta_j^{1/2}$ appears $|I| \, m^{|I|} / m$ times in the product $\prod_{v\in V} \eta_v^{1/2}$.
Since $|V| = m^{|I|}$, we combine the above with inequality~\eqref{Eqn:VolArg} to obtain
\begin{align*}
\log N(\varepsilon, \,\mathcal{E}, \,\|\cdot\|_{\ell_2})& \geq m^{|I|} \log\Big(\frac{1}{\varepsilon}\Big) - \frac{1}{2} |I|\, m^{|I|} \log a - \frac{|I|\,m^{|I|+2}}{8a^2}\\
& = m^{|I|} \log\Big(\frac{1}{\varepsilon\, a^{|I|/2}}\Big) - \frac{|I|\,m^{|I|+2}}{8a^2},
\end{align*}
By choosing
\begin{align*}
m \asymp \frac{a}{|I|} \log^{\frac12}\Big(\frac{1}{\varepsilon\, a^{|I|/2}}\Big),
\end{align*}
in the above display, we obtain
\begin{align*}
\log N(\varepsilon, \,\mathcal{E}, \,\|\cdot\|_{\ell_2}) \geq 
C\frac{a^{|I|}}{|I|^{|I|}}\log\Big(\frac{1}{\varepsilon\,  a^{|I|/2}}\Big)^{\frac{|I|+2}{2}}.
\end{align*}

\subsection{Proof of Theorem~\ref{Thm:HDRconsistency}}
As before, we will provide the proof only for the ordinary posterior. All the parts except for the verification of Assumption B2 are same as in the proof of Theorem 2.2 in \cite{norets2016} and Theorem~\ref{Thm:HNRconsistency} of the current paper.  The following Lemma guarantees the prior anti-concentration condition B2.
\begin{lems}
Under Assumption DGP and the prior distribution on $\theta$ in \S  \ref{Sec:DenReg},   for $I \supset I^\star$ and 
$\varepsilon_{n,I} = n^{-\frac{\beta}{2\beta + |I|+1}} \, (\log n)^t \wedge \sqrt{\frac{d_0\log p}{n}}$ with $t$ as defined in the statement of Theorem ~\ref{Thm:HDRconsistency}, we have
\begin{eqnarray}
\Pi( d_h^2 (f_0, p(\cdot | \cdot, \theta, m))  \leq M \varepsilon_{n, I}^2 )  \leq e^{- Hn \varepsilon_{n, I^\star}^2}
\end{eqnarray}
for some constants $M, H > 0$.  
\end{lems}
\begin{proof}
Without loss of generality, we assume $\mu$ to be $\mbox{Unif}(0,1)$.   By triangle inequality, we have for any $\theta^*$ in the parameter space
\begin{align*}
d_h^2(f_0, p(\cdot | \cdot, \theta, m)) \geq \mathrm{I} - \mathrm{II} - \mathrm{III} 
\end{align*}
where $\mathrm{I} =  d_H^2(p(\cdot |\theta, m) , p(\cdot |\theta^*, m) ), \mathrm{II} =     d_H^2(p(\cdot |\cdot, \theta^*, m), p(\cdot |\theta^*, m) ), \mathrm{III} = d_H^2(f_0, p(\cdot |\theta^*, m) )$.  
We will first provide an upper bound for $\mathrm{III}$ by choosing an appropriate $\theta^*$. For 
$\sigma_n = [\varepsilon_{n, I} / \log (1/ \varepsilon_{n, I}) ]^{1/\beta}$,
$\varepsilon$ defined in \eqref{eq:asnE0Dff0_Lf0},
a sufficiently small $\delta>0$,
$b$ and $\tau$ defined in \eqref{eq:asnf0_exp_tails},
$a_0 = \{ (8\beta + 4\varepsilon +16)/(b \delta)\}^{1/\tau}$,
$a_{\sigma_n} = a_0 \{\log (1/\sigma_n) \}^{1/\tau}$, 
and $b_1 > \max \{1, 1/ 2\beta \}$ satisfying $\varepsilon_{n, I}^{b_1}  \{  \log (1/ \varepsilon_{n, I}) \}^{5/4} \leq \varepsilon_{n, I}$,
the proof of Theorem 4 in \cite{shen2013adaptive} implies the following three claims.
First,
there exists a partition of $\{z \in \mathcal{Z}: ||z|| \leq a_{\sigma_n}\}$,
$\{U_j, j=1,\ldots,K\}$ such that for $j=1,\ldots,N$,
$U_j$ is a ball with diameter $\sigma_n \varepsilon_{n, I}^{2 b_1}$
and center $z_j = (x_j, y_j)$;
for $j=N+1,\ldots,K$,  
$U_j$ is a set with a diameter bounded above by $\sigma_n$;
$1 \leq N < K \leq C_2 \sigma_n^{-d} \{\log (1/ \varepsilon_{n, I}) \}^{d +d/\tau}$, 
where $C_2>0$ does not depend on $n$.
Second, there exist 
$\theta^\star = \{\mu_j^\star, \alpha_j^\star, j = 1,2,\ldots; \sigma_n\}$ with
$\alpha_j^\star=0$ for $j > N$, $\mu_j^\star=z_j$ for $j=1,\ldots,N$, and
$\mu_j^\star \in U_j$ for $j=N+1,\ldots,K$
 such that 
for $m=K =  \sigma_n^{-|I|} \{\log (1/ \varepsilon_{n, I}) \}^{|I| +|I|/\tau}$ and a positive constant $C_3$,
\begin{equation}
\label{eq:f0upsbeta}
d_H(f_0, p(\cdot |\theta^\star, m) ) \leq C_3 \sigma_n^\beta.
\end{equation}
Furthermore, since $m \varepsilon_{n, I} \to 0$ as $\beta > d_0$, one can choose the probabilities $\alpha_j^\star$ of $\theta^\star$  to be larger or equal to $\varepsilon_{n, I}/\delta$,
for some $0<\delta < 1$  since if not, we can add $\varepsilon_{n, I}/\delta$ to  $\alpha_j^\star$ and renormalize. 
We next prove that $\mathrm{II} \leq \mathrm{III}$.   Observe that 
\begin{eqnarray*}
\mathrm{II} = \int (\sqrt{g(x)} - 1)^2 dx
\end{eqnarray*}
where $g(x)$ is the marginal density of $x$ obtained from $p(\cdot |\theta^\star, m)$, i.e. $g(x) = \int  p(y, x|\theta^\star, m) dy$. 
Write $1= \int f_0(y |x) dy$.   Note that by H\"{o}lder's inequality,
\begin{eqnarray*}
\sqrt{g(x)} \geq  \int  \sqrt{p(y, x|\theta^\star, m) f_0(y |x)} dy
\end{eqnarray*}
Hence,
\begin{eqnarray*}
 \mathrm{II} = d_H^2(p(\cdot |\cdot, \theta^\star, m), p(\cdot |\theta^\star, m) )
&\leq& 2 -  2\int \int \sqrt{p(y, x|\theta^\star, m) f_0(y |x)} dy dx  = \mathrm{III}. 
\end{eqnarray*}

Finally, we derive a lower  bound for $\mathrm{I}$.  From  the proof of Theorem 3.1 of \cite{ho2016strong}, we obtain
for exact-fitted mixtures, the first-order identifiability condition (which is trivially satisfied for multivariate isotropic Gaussian kernels) suffices for obtaining that
\begin{eqnarray*}
d_H^2(p(\cdot |\theta, m) , p(\cdot |\theta^\star, m) )  \geq W_1^2(p(\cdot |\theta, m) , p(\cdot |\theta^\star, m))
\end{eqnarray*}
where $W_1$ is the Wasserstein distance of order $1$ and is given by 
\begin{eqnarray}\label{eq:wasser}
W_1(p(\cdot |\theta, m) , p(\cdot |\theta^\star, m)) =  \inf_q \sum_{j, j'} q_{jj'} \norm{\mu_j - \mu_{j'}^\star}.  
\end{eqnarray}
In \eqref{eq:wasser}, the infimum is taken over all joint probability distributions $q$ on $[1, \ldots , m]^2$ such that
when expressing $q$ as a $m \times m$ matrix under the marginal constraints
$\sum_j q_{jj'} = \alpha_{j'}$ and $\sum_{j'} q_{jj'} = \alpha_{j}^\star$. 
Hence it suffices to obtain a result like 
\begin{eqnarray*}
P\bigg( \inf_{q} \sum_{j=1}^mq_{jj} \norm{\mu_j - \mu_{j}^\star} < M\varepsilon_{n, I}\bigg) \leq e^{- n H \varepsilon_{n, I^\star}^2}.  
\end{eqnarray*}
for some constants $M, H > 0$.  
It  follows from \cite{ho2016strong} that $q_{jj} = \min \{\alpha_j, \alpha_j^\star \}$ if  $\|\mu_j - \mu_{j}^\star\|$ are close enough for all $j = 1, \ldots, m$.  Since $\alpha_j^\star \geq \varepsilon_{n, I}/\delta, \alpha_j > b/m$ by construction, and using the fact that 
$P(\sum_{j=1}^m  \|\mu_j - \mu_{j}^\star\|  \leq \epsilon) \asymp \exp\{-Cm \log (1/\epsilon)\}$ for sufficiently small $\epsilon > 0$ and $C >0$, 
we obtain 
\begin{eqnarray}
\Pi( d_h^2 (f_0, p(\cdot | \cdot, \theta, m))  \leq M \varepsilon_{n, I}^2 )  \leq e^{- Hn \varepsilon_{n, I^\star}^2}. 
\end{eqnarray}
Using the result in Theorem 2.2 in \cite{norets2016}, Assumption A2 is satisfied with
\begin{align*}
\varepsilon_{n,I} = n^{-\frac{\beta}{2\beta + |I| +1}} \, (\log n)^t \wedge \sqrt{\frac{d_0\log p}{n}},
\end{align*}
with $t $ as defined in the statement of \eqref{eq:modelrate}. 
\end{proof}

\appendix
\makeatletter   
 \renewcommand{\@seccntformat}[1]{Appendix~{\csname the#1\endcsname}:\hspace*{0.5em}}
 \makeatother

\section{Some results on eigen system of popular  covariance kernels}\label{App:eigen}
In this section, we develop a general technique to find the eigen sytem of commonly used stationary covariance kernels.   We will focus our attention only to 
covariance kernels on $[0, 1]^I \times [0, 1]^I$ which are separable in its coordinates.  To that end we start with a kernel  $K^a(x, y)$ on $[0, 1]\times [0,1]$ and construct a separable class of covariance kernels on $[0, 1]^I\times [0, 1]^I$  by considering product of univariate kernels $K_I^a(x, y)  =  \prod_{j \in I} K^a(x_j, y_j)$. Example include the Mat\'{e}rn class of  covariance kernels 
\begin{eqnarray}\label{eq:matern}
K^a(x, y) =  \frac{2^{1-\nu}}{\Gamma(\nu)} \big\{\sqrt{2\nu} \abs{x-y} \big\}^\nu  B_\nu\big\{ \sqrt{2\nu} a \abs{x-y}\big\},
\end{eqnarray}
where $B_\nu$ is the modified Bessel function of the second kind for $0<\nu \leq \infty$. 
\eqref{eq:matern} contains the squared exponential 
\begin{eqnarray}\label{eq:sqexp}
K^a_I(x,y)   =  e^{-a^2 (x - y)^2}, 
\end{eqnarray}
as a special case when $\nu =\infty$. To obtain the eigen system of $K^a_I(x,y)$, it 
is  enough to characterize the eigen system of $K^a(x,y)$ defined on $[0, 1]\times [0, 1]$.   
Let $\{\phi_j(x),\, j=0,1,2,\ldots\}$ denote the orthonormal basis functions of the one-dimensional function space $\bbL_2(U;\, [0,1])$ such that $\phi_j$ is the $j$th eigenfunction of $K$ corresponding to the $j$th largest eigenvalue $\eta_j$ (the existence is guaranteed by Mercer's theorem). 
The following lemma characterizes the eigensystem of a general stationary kernel $K(x,y)$ with respect to $\|\cdot\|_{U,2}$, which turns out to be closely related to Fourier series, and is interesting in its own right.  Since $K$ is stationary, we simply write $k(x-y) = K(x,y)$. 
\blems\label{Eqn:EigenSystem}
For any stationary kernel $K$ over $[0,1]^2$, if $f_K$ is its spectral density, then all eigenvalues of $K$ are 
\begin{align*}
\eta_0 = \int_{-1}^1 k(t)\, dt, \quad \eta_{2j-1} =\eta_{2j} = \int_{-1}^1 k(t) \, e^{i\, j\pi t}\, dt, \quad j\geq 1,
\end{align*}
where $\eta_0$ has multiplicity one, corresponding to the eigenfunction $\phi_0(x) =1$, and all $\eta_{2j-1} = \eta_{2j}$ ($j\geq 1$) has multiplicity two, corresponding to eigenfunctions
\begin{align*}
\phi_{2j-1}(x)=\sin(j\pi x), \quad\mbox{and}\quad \phi_{2j}(x)=\cos(j\pi x).
\end{align*}
\elems

\begin{proof}
It is obvious that $\{\phi_{j}:\, j\geq 0\}$ forms a complete orthonormal basis for $\mathcal{L}_2(U; \,[0,1])$. Now $k(\cdot)$ is an even function over $[-1,1]$, we can expand it via Fourier series as
\begin{align*}
k(t) = \sum_{u=0}^\infty \alpha_u \cos(u \pi t), \quad t\in[-1,1],
\end{align*}
where we used the fact that the Fourier coefficient for $\sin$'s are zero because $k$ is an even function, and
\begin{align*}
\alpha_u = \int_{-1}^1 k(t) \cos (u\pi t)dt = \int_{-1}^1 k(t) e^{i \, u\pi t}dt, \quad u\geq 1,
\end{align*}
where in the last step we used the fact that $K(t)$ is an even function.
\begin{align*}
h(\psi) = \frac{1}{2\pi}\int e^{i\,t\psi} k(t) dt.
\end{align*}
 Use the identity that for any $x, y$,
\begin{align*}
\cos(x-y) = \cos(x)\cos(y)+\sin(x)\sin(y),
\end{align*}
we obtain
\begin{align*}
k(x,y) =K(x-y) = \sum_{u=0}^\infty \alpha_u\big[  \cos(u\pi x)\cos(u\pi y)+\sin(u\pi x)\sin(u\pi y) \big], \quad x, y\in[0,1],
\end{align*}
implying the claimed result.
\end{proof}
Next, we apply Lemma  \ref{Eqn:EigenSystem}  to find eigen-values of squared exponential \eqref{eq:sqexp}
and Mat\'{e}rn  \eqref{eq:matern} covariance kernels.  

\setcounter{props}{1}
 
\begin{props}\label{Prop:squareexp}
If $K^a(x,y)$ is given by \eqref{eq:sqexp}, then 
\begin{align*}
\eta_0 \asymp a^{-1}, \quad \mbox{and}\quad \eta_{2j-1}=\eta_{2j} \asymp  a^{-1} e^{-j^2/a^2}, \qquad\mbox{for $j=1,\ldots, a^2$}.
\end{align*}
\end{props}
\begin{proof}
Let $h^a$ denote the spectral density of any one dimension stationary kernel $K$ (by Bochner's theorem), that is,
\begin{align*}
K^a(x,y) &= \int e^{-i \, t(x-y)}h(t) dt, \qquad\mbox{$x, y\in[0,1]$},\\
\mbox{or}\quad  h^a(t) &=\frac{1}{2\pi} \int e^{i \, t u} k^a(u) du, \qquad t \in\bbR.
\end{align*}
For the one-dimensional Gaussian kernel $K^a$, its spectral density $h^a$ is
\begin{align*}
h^a(t) = \frac{\sqrt{\pi} }{a}\exp\Big\{-\frac{t^2}{4a^2}\Big\}, \quad t\in\bbR.
\end{align*}
 Note that  
$$\int_{|t| \leq 1}  k^a(u)\, e^{i\, j\pi u}\, du  = h^a(j) -   \int_{|t|\geq 1} k^a(u)\, e^{i\, j\pi u}\, du.   $$ 
Using the tail bound for the standard normal distribution, we have 
$$\int_{|t|\geq 1} k^a(u)\, e^{i\, j\pi u}\, du \leq \int_{|t|\geq 1} k^a(u)\, du \asymp a^{-1} \exp\{-a^2\},$$
for $j=0,1,2,\ldots$. Combining this with Lemma~\ref{Eqn:EigenSystem}, we obtain the proof since for  $a\geq 2$, $h^a(j) \geq a^{-1} \exp\{-a^2\}$ for $j=0, 1, 2, \ldots, a^2$.
\end{proof}

\begin{props}
If $K_a(x,y)$ is given by \eqref{eq:matern}, then   
\begin{align}
\lambda_0 \asymp  a^{-1},  \quad  \lambda_{2j-1}=\lambda_{2j} \asymp a^{-1} (1+ j^2/a^2)^{-(\nu + 1/2)}, j=1,\ldots, n^\kappa \label{eq:matbasis}
\end{align}
for any $\kappa > 0$, provided $a \geq C(\nu, \kappa) \log n$ for a constant $C(\nu, \kappa) > 0$ depending on $\nu$ and $\kappa$.  
\end{props}
\begin{proof}
The spectral density of the M\'{a}tern kernel is given by 
\begin{align*}
h^a(\psi) := \frac{1}{2\pi}\int e^{i\,t\psi} k^a(t) dt = C \frac{1}{a} \Big( 1  + \frac{\psi^2}{a^2}  \Big)^{-(\nu + 1/2)}. 
\end{align*}
We now argue that by choosing $a \ge  C\log n$ for a suitably large constant $C$, one can ensure that 
$\int_{|u| \leq 1}  k^a(u)\, e^{i\, j\pi u}\, du$ is of the same order as $\int_{-\infty}^{\infty}  k^a(u)\, e^{i\, j\pi u}\, du$. 
This is true, since
\begin{eqnarray*}
\int_{|u|\geq 1} k^a(u)\, e^{i\, j\pi u}\, du \leq  \int_{|u|\geq 1} k^a(u)\, du =  C_1\int_{|u|\geq C_2a} t^{\nu} B_{\nu}(t)  \asymp  e^{-C_3a}
\end{eqnarray*}
for $j=0,1,2,\ldots$,  
where the last inequality follows from Theorem 2.5 and (A.5) of \cite{gaunt2014inequalities}.   
Since $e^{-C_3a}$ can be made smaller than $\lambda_{n^{\kappa}}$ by choosing $a \geq C(\nu, \kappa) \log n$, we can estimate $\lambda_{2j-1} = \lambda_{2j} \asymp h^a(j \pi)$, delivering the proof of the Proposition. 
\end{proof}


\bibliographystyle{plain}

\bibliography{MS}

\begin{thebibliography}{10}

\bibitem{Bar:2007}
Peter~L. Bartlett.
\newblock Fast rates for estimation error and oracle inequalities for model
  selection.
\newblock {\em Econometric Theory}, 24, 2008.

\bibitem{bartlett2005}
Peter~L. Bartlett, Olivier Bousquet, and Shahar Mendelson.
\newblock Local rademacher complexities.
\newblock {\em Ann. Statist.}, 33:1497--1537, 2005.

\bibitem{Bartlett:2003}
Peter~L. Bartlett and Shahar Mendelson.
\newblock Rademacher and {G}aussian complexities: Risk bounds and structural
  results.
\newblock {\em J. Mach. Learn. Res.}, 3:463--482, 2003.

\bibitem{berger1996intrinsic}
James~O Berger and Luis~R Pericchi.
\newblock The intrinsic bayes factor for model selection and prediction.
\newblock {\em Journal of the American Statistical Association},
  91(433):109--122, 1996.

\bibitem{bhattacharya2014anisotropic}
Anirban Bhattacharya, Debdeep Pati, and David Dunson.
\newblock Anisotropic function estimation using multi-bandwidth gaussian
  processes.
\newblock {\em Annals of statistics}, 42(1):352, 2014.

\bibitem{bhattacharya2016bayesian}
Anirban Bhattacharya, Debdeep Pati, and Yun Yang.
\newblock Bayesian fractional posteriors.
\newblock {\em arXiv preprint arXiv:1611.01125}, 2016.

\bibitem{Bousquet:2002}
Olivier Bousquet and Andr{\'e} Elisseeff.
\newblock Stability and generalization.
\newblock {\em J. Mach. Learn. Res.}, 2:499--526, 2002.

\bibitem{castillo2015bayesian}
Isma{\"e}l Castillo, Johannes Schmidt-Hieber, Aad Van~der Vaart, et~al.
\newblock Bayesian linear regression with sparse priors.
\newblock {\em The Annals of Statistics}, 43(5):1986--2018, 2015.

\bibitem{comminges2012}
L.~Comminges and A.~S. Dalalyan.
\newblock Tight conditions for consistency of variable selection in the context
  of high dimensionality.
\newblock {\em Ann. Statist.}, 40:2667--2696, 2012.

\bibitem{gaunt2014inequalities}
Robert~E Gaunt.
\newblock Inequalities for modified bessel functions and their integrals.
\newblock {\em Journal of Mathematical Analysis and Applications},
  420(1):373--386, 2014.

\bibitem{ghosal2000}
Subhashis Ghosal, Jayanta~K. Ghosh, and Aad~W. van~der Vaart.
\newblock Convergence rates of posterior distributions.
\newblock {\em Ann. Statist.}, 28(2):500--531, 2000.

\bibitem{ghosal2008}
Subhashis Ghosal, Jüri Lember, and Aad van~der Vaart.
\newblock Nonparametric bayesian model selection and averaging.
\newblock {\em Electron. J. Statist.}, 2:63--89, 2008.

\bibitem{ghosal2007}
Subhashis Ghosal and Aad van~der Vaart.
\newblock Convergence rates of posterior distributions for noniid observations.
\newblock {\em Ann. Statist.}, 35(1):192--223, 2007.

\bibitem{ho2016strong}
Nhat Ho, XuanLong Nguyen, et~al.
\newblock On strong identifiability and convergence rates of parameter
  estimation in finite mixtures.
\newblock {\em Electronic Journal of Statistics}, 10(1):271--307, 2016.

\bibitem{johnson2012nonlocal}
V.E. Johnson and D.~Rossell.
\newblock On the use of non-local prior densities in {B}ayesian hypothesis
  tests.
\newblock {\em Journal of the Royal Statistical Society: Series B (Statistical
  Methodology)}, 72(2):143--170, 2010.

\bibitem{koltchinskii2006}
Vladimir Koltchinskii.
\newblock Local {R}ademacher complexities and oracle inequalities in risk
  minimization.
\newblock {\em Ann. Statist.}, 34(6):2593--2656, 2006.

\bibitem{kuelbs1993metric}
James Kuelbs and Wenbo~V Li.
\newblock Metric entropy and the small ball problem for gaussian measures.
\newblock {\em Journal of Functional Analysis}, 116(1):133--157, 1993.

\bibitem{lecam1973}
L.~LeCam.
\newblock Convergence of estimates under dimensionality restrictions.
\newblock {\em Ann. Statist.}, 1:38--53, 1973.

\bibitem{Lever:2013}
Guy Lever, Fran\c{c}ois Laviolette, and John Shawe-Taylor.
\newblock Tighter {PAC-B}ayes bounds through distribution-dependent priors.
\newblock {\em Theor. Comput. Sci.}, 473:4--28, 2013.

\bibitem{li2001gaussian}
Wenbo~V Li and Q-M Shao.
\newblock Gaussian processes: inequalities, small ball probabilities and
  applications.
\newblock {\em Handbook of Statistics}, 19:533--597, 2001.

\bibitem{martin2016optimal}
Ryan Martin and Stephen~G Walker.
\newblock Optimal bayesian posterior concentration rates with empirical priors.
\newblock {\em arXiv preprint arXiv:1604.05734}, 2016.

\bibitem{McAllester1998}
D.~McAllester.
\newblock Some {PAC}-b{a}yesian theorems.
\newblock In {\em Anual Conference on Computational Learning Theory}, pages
  230--234, 1998.

\bibitem{McAllester1999}
D.~McAllester.
\newblock {PAC}-{B}ayesian model averaging.
\newblock In {\em Anual Conference on Computational Learning Theory}, pages
  164--170, 1999.

\bibitem{Narisetty2014}
N.~Narisetty and X.~He.
\newblock Bayesian variable selection with shrinking and diffusing priors.
\newblock {\em Annals of Statistics}, 42:789--817, 2014.

\bibitem{norets2016}
Andriy Norets and Debdeep Pati.
\newblock Adaptive bayesian estimation of conditional densities.
\newblock {\em Econometric Theory}, pages 1--33, 007 2016.

\bibitem{Oneto2016}
Luca Oneto, Alessandro Ghio, Sandro Ridella, and Davide Anguita.
\newblock Global rademacher complexity bounds: From slow to fast convergence
  rates.
\newblock {\em Neural Processing Letters}, 43:567--602, 2016.

\bibitem{Raskutti2012}
G.~Raskutti, M.~Wainwright, and B.~Yu.
\newblock Minimax-optimal rates for sparse additive models over kernel classes
  via convex programming.
\newblock {\em Journal of Machine Learning Research}, 13:389--427, 2012.

\bibitem{savitsky2011variable}
T.~Savitsky, M.~Vannucci, and N.~Sha.
\newblock Variable selection for nonparametric gaussian process priors: Models
  and computational strategies.
\newblock {\em Statistical science: a review journal of the Institute of
  Mathematical Statistics}, 26(1):130, 2011.

\bibitem{Shang2011}
Z.~Shang and M.~Clayton.
\newblock Consistency of {B}ayesian linear model selection with a growing
  number of parameters.
\newblock {\em Journal of Statistical Planning and Inference}, 141:3463--3474,
  2011.

\bibitem{shen2013adaptive}
Weining Shen, Surya~T Tokdar, and Subhashis Ghosal.
\newblock Adaptive bayesian multivariate density estimation with dirichlet
  mixtures.
\newblock {\em Biometrika}, 100(3):623--640, 2013.

\bibitem{Tibshirani1996}
R.~Tibshirani.
\newblock Regression shrinkage and selection via the lasso.
\newblock {\em Journal of the Royal Statistical Society (Series B)},
  58:267--288, 1996.

\bibitem{vdV00}
A.~W. van~der Vaart.
\newblock {\em Asymptotic statistics}.
\newblock Cambridge Series in Statistical and Probabilistic Mathematics.
  Cambridge University Press, 1998.

\bibitem{vandervaart2008}
A.~W. van~der Vaart and J.~H. van Zanten.
\newblock {\em Reproducing kernel Hilbert spaces of Gaussian priors}, volume
  Volume 3 of {\em Collections}, pages 200--222.
\newblock Institute of Mathematical Statistics, Beachwood, Ohio, USA, 2008.

\bibitem{vandervaart2009}
A.~W. van~der Vaart and J.~H. van Zanten.
\newblock Adaptive bayesian estimation using a gaussian random field with
  inverse gamma bandwidth.
\newblock {\em Ann. Statist.}, 37:2655--2675, 2009.

\bibitem{van2014renyi}
Tim Van~Erven and Peter Harremos.
\newblock R{\'e}nyi divergence and kullback-leibler divergence.
\newblock {\em IEEE Transactions on Information Theory}, 60(7):3797--3820,
  2014.

\bibitem{Vapnik1998}
Vladimir~N. Vapnik.
\newblock {\em Statistical learning theory}.
\newblock Wiley-Interscience, 1998.

\bibitem{Wainwright2009}
M.~Wainwright.
\newblock Information-theoretic limits on sparsity recovery in the
  high-dimensional and noisy setting.
\newblock {\em IEEE Transactions on Information Theory,}, 55:5728--5741, 2009.

\bibitem{walker2001bayesian}
Stephen Walker and Nils~Lid Hjort.
\newblock On bayesian consistency.
\newblock {\em Journal of the Royal Statistical Society: Series B (Statistical
  Methodology)}, 63(4):811--821, 2001.

\bibitem{wegkamp2003}
Marten Wegkamp.
\newblock Model selection in nonparametric regression.
\newblock {\em Ann. Statist.}, 31:252--273, 2003.

\bibitem{Yang2015}
M.~Jordan Y.~Yang, M.~Wainwright.
\newblock On the computational complexity of high-dimensional {B}ayesian
  variable selection.
\newblock {\em arXiv preprint: 1505.07925}, 2015.

\bibitem{yang2015b}
Yun Yang and Surya~T. Tokdar.
\newblock Minimax-optimal nonparametric regression in high dimensions.
\newblock {\em Ann. Statist.}, 43:652--674, 2015.

\bibitem{zou2010nonparametric}
F.~Zou, H.~Huang, S.~Lee, and I.~Hoeschele.
\newblock Nonparametric bayesian variable selection with applications to
  multiple quantitative trait loci mapping with epistasis and gene--environment
  interaction.
\newblock {\em Genetics}, 186(1):385--394, 2010.

\end{thebibliography}
\end{document}